\documentclass[12pt]{amsart}
\usepackage{ amsmath, amsthm, amsfonts, amssymb, color}
 \usepackage{mathrsfs}
\usepackage{amsfonts, amsmath}
 \usepackage{amsmath,amstext,amsthm,amssymb,amsxtra}
 \usepackage{txfonts} 
 \usepackage[colorlinks, citecolor=blue,pagebackref,hypertexnames=false]{hyperref}
 \allowdisplaybreaks
 \usepackage{pgf,tikz}

\begin{document}

 \baselineskip 16.6pt
\hfuzz=6pt

\widowpenalty=10000

\newtheorem{cl}{Claim}
\newtheorem{theorem}{Theorem}[section]
\newtheorem{proposition}[theorem]{Proposition}
\newtheorem{coro}[theorem]{Corollary}
\newtheorem{lemma}[theorem]{Lemma}
\newtheorem{definition}[theorem]{Definition}
\newtheorem{assum}{Assumption}[section]
\newtheorem{example}[theorem]{Example}
\newtheorem{remark}[theorem]{Remark}
\renewcommand{\theequation}
{\thesection.\arabic{equation}}

\def\SL{\sqrt H}

\newcommand{\mar}[1]{{\marginpar{\sffamily{\scriptsize
        #1}}}}

\newcommand{\as}[1]{{\mar{AS:#1}}}

\newcommand\R{\mathbb{R}}
\newcommand\RR{\mathbb{R}}
\newcommand\CC{\mathbb{C}}
\newcommand\NN{\mathbb{N}}
\newcommand\ZZ{\mathbb{Z}}
\newcommand\HH{\mathbb{H}}
\def\RN {\mathbb{R}^n}
\renewcommand\Re{\operatorname{Re}}
\renewcommand\Im{\operatorname{Im}}

\newcommand{\mc}{\mathcal}
\newcommand\D{\mathcal{D}}
\def\hs{\hspace{0.33cm}}
\newcommand{\la}{\alpha}
\def \l {\alpha}
\newcommand{\eps}{\varepsilon}
\newcommand{\pl}{\partial}
\newcommand{\supp}{{\rm supp}{\hspace{.05cm}}}
\newcommand{\x}{\times}
\newcommand{\lag}{\langle}
\newcommand{\rag}{\rangle}

\newcommand\wrt{\,{\rm d}}

\title[]{Quantitative weighted estimates for rough singular integrals\\ on homogeneous groups}

\author{Zhijie Fan and Ji Li}

 \address{Zhijie Fan, Department of Mathematics, Sun Yat-sen
University, Guangzhou, 510275, P.R. China}
\email{fanzhj3@mail2.sysu.edu.cn}

\address{Ji Li, Department of Mathematics, Macquarie University, NSW, 2109, Australia}
\email{ji.li@mq.edu.au}

  \date{\today}

\begin{abstract}
In this paper, we study weighted $L^{p}(w)$ boundedness ($1<p<\infty$ and $w$ a Muckenhoupt $A_{p}$ weight) of singular integrals with homogeneous convolution kernel
$K(x)$
on an arbitrary homogeneous group $\mathbb H$ of dimension $\mathbb{Q}$, {under the assumption that $K_0$, the restriction of $K$ to the unit annulus, is mean zero and $L^{q}$ integrable  for some $q_{0}<q\leq \infty$,} where $q_{0}$ is a fixed constant depending on $w$.  We obtain a quantitative weighted bound, which is consistent with the one obtained by Hyt\"onen--Roncal--Tapiola in the Euclidean setting, for this operator on $L^{p}(w)$.
Comparing to the previous results in the Euclidean setting, our assumptions on the kernel and on the underlying space are weaker.
Moreover, we investigate the
quantitative weighted bound for the bi-parameter rough singular integrals on product homogeneous Lie groups.
\end{abstract}

\maketitle


\section{Introduction}\label{se1}
Let $\HH=\RR^{n}$ be a homogeneous group (see \cite{FoSt,s93}), which is a nilpotent Liegroup with multiplication, inverse, dilation, and norm structures
$(x,y)\mapsto xy, \ \ x\mapsto x^{-1},\ \ (t,x)\mapsto t\circ x,\ \ x\mapsto \rho(x)$
for $x,y\in\HH$, $t>0$. The multiplication and inverse operations are polynomials and form a group with identity 0, the dilation structure preserves the group operations and is given in coordinates by
\begin{align*}
t\circ (x_{1},\ldots,x_{n})=(t^{\alpha_{1}}x_{1},\ldots,t^{\alpha_{n}}x_{n})
\end{align*}
for some constants $0<\alpha_{1}\leq \alpha_{2}\leq\ldots\leq \alpha_{n}$. Besides, $\rho(x):=\max\limits_{1\leq j\leq n}\{|x_{j}|^{1/a_{j}}\}$ is a norm associated to the dilation structure. We call $n$ the Euclidean dimension of $\HH$, and the quantity $\mathbb{Q}=\sum_{j=1}^{n}\alpha_{j}$ the homogeneous dimension of $\HH$, respectively.

We now recall the notion of homogeneous singular integrals on homogeneous group. Let $\Sigma:=\{x\in \HH:\rho(x)=1\}$ and $K$ be a homogeneous convolution kernel on $\HH$, so that
\begin{align*}
K(x)=\frac{\Omega(\rho(x)^{-1}\circ x)}{\rho(x)^{\mathbb{Q}}} \quad{\rm with}\quad \int_{\Sigma}\Omega(\theta)dS(\theta)=0.
\end{align*}
The homogeneous singular integral operator $T$ is defined  initially for $f\in C_{0}^{\infty}(\HH)$ as follows
\begin{align*}
T(f)(x):={\rm p.v.}\int_{\HH}K(y)f(y^{-1}x)dy=\lim\limits_{\substack{\varepsilon\rightarrow 0\\R\rightarrow \infty}}\int_{\varepsilon<\rho(y)<R}K(y)f(y^{-1}x)dy.
\end{align*}

The study of rough singular integral operators dates back to Calder\'{o}n and Zygmund's work \cite{CZ0,CZ}. It is well-known that when $\HH$ is an isotropic Euclidean space, Calder\'{o}n and Zygmund \cite{CZ} used the method of rotations to show that if $\Omega\in L{\rm log}L(\mathbb{S}^{n-1})$, then $T$ is bounded on $L^{p}(\mathbb{R}^{n})$ for all $1<p<\infty$. Later it was shown by Christ \cite{Christ1,Christ2},  Hofmann \cite{Hof} and Seeger \cite{Se} that such operators are of weak-type $(1,1)$ and by Tao \cite{Tao} that the underlying space can be generalized to homogeneous group $\HH$. There are also many other important progress on rough singular integral operators (see for example \cite{CD,DL,DLu,DR,FP, GS,  LMW}).

Furthermore, there has been numerous work on weighted inequalities of singular integral with rough kernels, see for example \cite{D,FPY,Lerner4,LPRR} and the references therein for its development and applications.
Recently, the sharp weight inequalities for standard Calder\'{o}n--Zygmund operators was proved by Hyt\"{o}nen  \cite{H} via constructing the representation theorem, which gives the following weighted $L^{p}$ bound with sharp dependence on $[w]_{A_{p}}$.
\begin{align*}
\|Tf\|_{L^{p}(w)}\leq C_{p,T}[w]_{A_{p}}^{\max\{1,1/(p-1)\}}\|f\|_{L^{p}(w)},\ \ 1<p<\infty.
\end{align*}
Besides, Lerner \cite{Lerner1,Lerner2} and Lacey \cite{La} gave alternative approaches to this result. Then a natural question arises: can we also obtain a sharp weight bound for rough homogeneous singular integral operators? We point out that this topic has been studied intensively especially in the last three years (the pointwise version originated in \cite{La}) with the key tool sparse domination, see for example \cite{CLRT,CCDO,HRT,La01,La02,Ler1,Lerner3,PRR}. {Among these results, we would like to highlight that Hyt\"{o}nen--Roncal--Tapiola \cite{HRT} first quantitatively proved that if $\Omega\in L^{\infty}(\mathbb{S}^{n-1})$, then
\begin{align}\label{Q W bd for T}
\|T\|_{L^{p}(w)\rightarrow L^{p}(w)}\leq C_{n,p}\|\Omega\|_{L^{\infty}}\{w\}_{A_{p}}(w)_{A_{p}}.
\end{align}
In particular,
\begin{align*}
\|T\|_{L^{2}(w)\rightarrow L^{2}(w)}\leq C_{n}\|\Omega\|_{L^\infty}[w]_{A_2}^{2},
\end{align*}
(For the definitions of $\{w\}_{A_{p}}$ and $(w)_{A_{p}}$, we refer the readers to Section 2). Different proofs of the quantitative bound of $T$  (as in \eqref{Q W bd for T}) via a sparse domination principle were obtained by Conde-Alonso, Culiuc, Di Plinio and Ou \cite{CCDO}, and by Lerner \cite{Ler1}.}

Inspired by Tao's work \cite{Tao}, Sato \cite{sato} extended part of the classical results related to singular integral to homogeneous group. He obtained the $L^{p}(w)$ boundedness for rough homogeneous singular integral operators under the assumption that $w\in A_{p}$ for some $1<p<\infty$ and $\Omega\in L^{\infty}(\Sigma)$. However, It is still unclear that whether a quantitative weight bound can be obtained in this setting and that whether the condition $\Omega\in L^{\infty}(\Sigma)$ can be weakened.

The purpose of this paper is to address these points. As in \cite{Tao}, let $K_{0}$ be the restriction of $K$ to the annulus $\Sigma_{0}=\{x\in\HH: 1\leq \rho(x)\leq 2\}$, then it is clear that $\Omega\in L^{q}(\Sigma)$ implies $K_{0}\in L^{q}(\Sigma_{0})$ for the same $q\in(1,\infty]$. Our main result is the following theorem.
\begin{theorem}\label{main}
Let $1<p<\infty$, $w\in A_{p}$. Suppose that $K_{0}$ has mean zero and there exists a constant $c_{\mathbb{Q},p}>0$ such that $K_{0}\in L^{q}$ for some $q>c_{\mathbb{Q},p}(w)_{A_{p}}$, then
\begin{align*}
\|T\|_{L^{p}(w)\rightarrow L^{p}(w)}\leq C_{\mathbb{Q},p,q}\|K_{0}\|_{q}\{w\}_{A_p}(w)_{A_p}.
\end{align*}
In particular,
\begin{align*}
\|T\|_{L^{2}(w)\rightarrow L^{2}(w)}\leq C_{\mathbb{Q},q}\|K_{0}\|_{q}[w]_{A_2}^{2},
\end{align*}
for some constants $C_{\mathbb{Q},p,q}$ and $C_{\mathbb{Q},q}$ independent of $w$.
\end{theorem}
Comparing with the previous closely related results, we point out that the weight bound $\{w\}_{A_p}(w)_{A_p}$ we obtained is consistent with that obtained in \cite{HRT}. It is still unknown that whether this is sharp, but it is the best known quantitative result for this class of operators.


To show Theorem \ref{main}, we borrow the idea from \cite{HRT} to divide the proof into two steps:


\smallskip
$\bullet$
In the first step, noting that the Fourier transform is not applicable in general homogeneous groups,
we  decompose the operator $T$ into a summation of Dini-type Calder\'{o}n-Zygmund operators $\tilde{T}_{j}^{N}$ defined by \eqref{b2}, which is slightly different from \cite{HRT} but is crucial to implement the iterated $TT^{*}$ argument from \cite{Tao}.
Then we combine Cotlar-Knapp-Stein Lemma with a key $L^{2}$ estimate originated from \cite{Tao} (later extended by \cite{sato}) to show the unweighted $L^{2}$ estimate for $\tilde{T}_{j}^{N}$, that is, for any $j\geq 1$,
\begin{align*}
\|\tilde{T}_{0}(f)\|_{2}\lesssim 2^{-\alpha j}\|K_{0}\|_{q}\|f\|_{2},\hskip1cm
\|\tilde{T}_{j}^{N}(f)\|_{2}\lesssim 2^{-\alpha N(j-1)}\|K_{0}\|_{q}\|f\|_{2}.
\end{align*}

\smallskip
$\bullet$
In the last step, we prove a quantitative $L^{p}$ weighted inequality and a quantitative good unweighted $L^{p}$ estimate for $\tilde{T}_{j}^{N}$, both of which contains an extra bad factor $2^{\frac{N(j)\mathbb{Q}}{q}}$, that is,
\begin{align*}
&\|\tilde{T}_{j}^{N}(f)\|_{L^{p}(w)}\lesssim 2^{\frac{N(j)\mathbb{Q}}{q}}(1+N(j))\|K_{0}\|_{q}\{w\}_{A_{p}}\|f\|_{L^{p}(w)},\\
&\|\tilde{T}_{j}^{N}(f)\|_{L^{p}} \lesssim2^{-\beta_{p}N(j-1)}2^{\frac{N(j)\mathbb{Q}}{q}}(1+N(j))\|K_{0}\|_{q}\|f\|_{L^{p}}.
\end{align*}
Finally, our Theorem \ref{main} can be proven after choosing appropriate $N(j)$ and repeating a standard argument of interpolation theorem with change of measures.

As a direct application, we obtain the quantitative weighted bound for the rough singular integrals studied by Sato \cite{sato}.
To state our result, we first recall some notations introduced in \cite{sato}.

For $q\geq 1$, let $d_{q}$ denote the collection of measurable functions $h$ on $\mathbb{R}_{+}=\{t\in\mathbb{R}:t>0\}$ satisfying
$
\|h\|_{d_{q}}=\sup\limits_{j\in\mathbb{Z}}\left(\int_{2^{j}}^{2^{j+1}}|h(t)|^{q}\frac{dt}{t}\right)^{1/q}<\infty,
$
We define $\|h\|_{d_{\infty}}:=\|h\|_{L^{\infty}(\mathbb{R}_{+})}$. Besides, for $t\in(0,1]$, let
$
u(h,t):=\sup\limits_{|s|<tR/2}\int_{R}^{2R}|h(r-s)-h(r)|\frac{dr}{r},
$
where the supremum is taken over all $s$ and $R$ such that $|s|<tR/2$. For $\eta>0$, let $\Lambda^{\eta}$ denote the family of functions $h$ such that
$\|h\|_{\Lambda^{\eta}}:=\sup\limits_{t\in (0,1]}t^{-\eta}u(h,t)<\infty.
$
Define $\Lambda_{q}^{\eta}:=d_{q}\cap\Lambda^{\eta}$ and set $\|h\|_{\Lambda_{q}^{\eta}}:=\|h\|_{d_{q}}+\|h\|_{\Lambda^{\eta}}$ for $h\in \Lambda_{q}^{\eta}$.
Consider
\begin{align*}
T(f)(x)={\rm p.v.}\ f\ast L(x)={\rm p.v.}\int_{\HH}f(y)L(y^{-1}x)dy,
\end{align*}
where $L(x):=h(\rho(x))K(x)$ and $K$ is defined in Section \ref{se1}. Then we have
\begin{theorem}\label{main2}
Let $1<p<\infty$, $w\in A_{p}$. Suppose that $K_{0}$ has mean zero and there exists a constant $c_{\mathbb{Q},p}>0$ such that $K_{0}\in L^{q}$ for some $q>c_{\mathbb{Q},p}(w)_{A_{p}}$. Suppose that $h\in \Lambda_{q}^{\eta/q^{\prime}}$ for some $\eta>0$, then
\begin{align*}
\|T\|_{L^{p}(w)\rightarrow L^{p}(w)}\leq C_{\mathbb{Q},p,q}\|K_{0}\|_{q}\|h\|_{\Lambda_{q}^{\eta/q^{\prime}}}\{w\}_{A_p}(w)_{A_p}.
\end{align*}
for some constant $C_{\mathbb{Q},p,q}$ independent of $w$.
\end{theorem}

We also have an investigation on the quantitative weighted estimate for bi-parameter rough singular integrals on the product homogeneous groups
$\HH_{1}\times\HH_{2}$. Recall that the Euclidean version was introduced by R. Fefferman \cite[page 198]{RF1981} where $\Omega$ is Lipschitz, and later studied by Duoandikoetxea \cite{Duo1986}. See also \cite{AP,AAP,sato2} for previous works about rough singular integrals on product of Euclidean spaces.
Consider the singular integral
\begin{align*}
Tf(x,y)={\rm p.v.}f\ast K(x,y)={\rm p.v.}\int_{\HH_{1}\times\HH_{2}}f(xu^{-1},yv^{-1})K(u,v)dudv,
\end{align*}
where $K(u,v)$ satisfies $$K(t_{1}\circ_{1}u,t_{2}\circ_{2}v)=t_{1}^{-\mathbb{Q}_{1}}t_{2}^{-\mathbb{Q}_{2}}K(u,v),$$
for all $t_{i}\in\mathbb{R}_{+}$ and $(u,v)\in\HH_{1}\times\HH_{2}$. For $i=1,2$, let $D_{0}^{(i)}=\{x_{i}\in\HH_{i}:1\leq \rho_{i}(x_{i})\leq 2\}$ and $D_{0}=D_{0}^{(1)}\times D_{0}^{(2)}$. Denote $x=(x_{1},x_{2})\in\HH_{1}\times\HH_{2}$ and $K^{0}(x)=K(x)\chi_{D_{0}}(x)$. In the bi-parameter setting, we also abuse the notation $w\in A_{p}$ to denote that $w$ is a product $A_{p}$ weight. We now state our result in the bi-parameter setting. For the sake of simplicity, we refer the readers to Section \ref{bi} for all details of definitions and notations.
\begin{theorem}\label{main3}
Let $w\in A_{2}$. Suppose that $K^{0}$ satisfies
\begin{align*}
\int_{D_{0}^{(1)}}K(u,v)du=\int_{D_{0}^{(2)}}K(u,v)dv=0, \ \ {\rm for}\ {\rm all}\ (u,v)\in D_{0},
\end{align*}
and there exists a constant $c_{\mathbb{Q}_1,\mathbb{Q}_2}>0$ such that $K^{0}\in L^{q}(D_{0})$ for some $q>c_{\mathbb{Q}_1,\mathbb{Q}_{2}}(w)_{A_{2}}$, then
\begin{align*}
\|T\|_{L^{2}(w)\rightarrow L^{2}(w)}\leq C_{\mathbb{Q}_1,\mathbb{Q}_2,q}\max\{\|K^{0}\|_{q},\|K^{0}\|_{q}^{2}\}[w]_{A_2}^{12}[w]_{A_\infty}^2
\end{align*}
for some constant $C_{\mathbb{Q}_1,\mathbb{Q}_2,q}$ independent of $w$.
\end{theorem}

Based on the framework of the proof of Theorem \ref{main}, the key idea to prove Theorem \ref{main3} is
to decompose the bi-parameter rough singular integral $T$
into a summation of bi-parameter Dini-type Calder\'on--Zygmund operators $\tilde T_{j_1,j_2}^N$ with the modified Dini-1 condition as in \cite[Section 5]{AMV} and with the cancellation on the kernel $K_{j_1,j_2}^N$.

Then we have the standard bi-parameter representation theorem for each $\tilde T_{j_1,j_2}^N$, which, together with \cite[Corollary 3.2]{BP} and the sparse domination for the Shifted Square Function \cite[Section 5]{BP}, gives
Theorem \ref{main3}.

It is not clear whether the quantitative estimate appearing in Theorem \ref{main3} can be pushed down further using our methods.

This paper is organized as follows. In Section \ref{preliminariessec} we provide the preliminaries, including the fundamental properties of the Muckenhoupt $A_{p}$ weights and the definition of Calder\'{o}n-Zygmund operators with Dini-continuous kernel. In Section \ref{keey}, we prove our main result Theorem \ref{main}. In Section \ref{satosec}, we prove Theorem \ref{main2}, the quantitative weighted bound of the rough singular integrals studied by Sato \cite{sato}. In the last section we investigate the quantitative weighted bound in the bi-parameter setting.

\bigskip
\section{Preliminaries}\label{preliminariessec}
\setcounter{equation}{0}
\subsection{Muckenhoupt $A_{p}$ weights}
To begin with, we define a left-invariant quasi-distance $d$ on $\HH$ by $d(x,y)=\rho(x^{-1}y)$, which means that
there exists a constant $A_{0}\geq 1$ such that for any $x,y,z\in\HH$,
\begin{align*}
d(x,y)\leq A_{0}[d(x,z)+d(z,y)].
\end{align*}
Next, let $B(x,r):=\{y\in\HH :d(x,y)<r\}$ be the open ball with center $x\in\HH$ and radius $r>0$.

We denote the average of a function $f$ over a ball $B$ by
\begin{align*}
\langle f\rangle_{B}=\fint_{B}fdx=\frac{1}{|B|}\int_{B}f(x)dx,
\end{align*}
where $|B|$ denotes the Lebesgue measure of $B$.
\begin{definition}
Let $w(x)$ be a nonnegative locally integrable function
  on~$\HH$. For $1 < p < \infty$, we
  say that $w$ is an $A_p$ \emph{weight}, written $w\in
  A_p$, if
\begin{align*}
[w]_{A_{p}}:=\sup\limits_{B}\left(\fint_{B}wdx\right)\left(\fint_{B}\left(\frac{1}{w}\right)^{1/(p-1)}dx\right)^{p-1}<\infty,
\end{align*}
where the supremum is taken over all balls~$B\subset \HH$. The quantity $[w]_{A_p}$ is called the \emph{$A_p$~constant
  of~$w$}.
  For $p = 1$, if $M(w)(x)\leq w(x)$ for a.e. $x\in \HH$, then we say that $w$ is an $A_1$ \emph{weight},
  written $w\in A_1$, where $M$ denotes the Hardy-Littlewood maximal function. Besides, let $A_\infty := \cup_{1\leq p<\infty} A_p$ and we have
\begin{align*}
[w]_{A_{\infty}}:=\sup\limits_{B}\left(\fint_{B}wdx\right){\rm exp}\left(\fint_{B}{\rm log}\left(\frac{1}{w}\right)dx\right)<\infty.
\end{align*}
\end{definition}
In order to state our weighted estimates much more efficiently, we recall the following variants of the weight characteristic (see for example \cite{HRT}):
\begin{align*}
\{w\}_{A_{p}}:=[w]_{A_{p}}^{1/p}{\rm max}\{[w]_{A_{\infty}}^{1/p^{\prime}},[w^{1-p^{\prime}}]_{A_{\infty}}^{1/p}\},
\quad (w)_{A_{p}}:={\rm max}\{[w]_{A_{\infty}},[w^{1-p^{\prime}}]_{A_{\infty}}\}.
\end{align*}
\begin{lemma}\label{wei2}
Let $1<p<\infty$, and $w\in A_{p}$. Then there exists a constant $c_{\mathbb{Q}}$ small enough such that for every $0<\delta\leq c_{\mathbb{Q}}/(w)_{A_{p}}$, we have that $w^{1+\delta/2}\in A_{p}$ and
\begin{align*}
(w^{1+\delta/2})_{A_{p}}\leq C_{\mathbb{Q}}(w)_{A_{p}}^{1+\delta/2}, \ \ \{w^{1+\delta/2}\}_{A_{p}}\leq C_{\mathbb{Q}}\{w\}_{A_{p}}^{1+\delta/2}.
\end{align*}
\end{lemma}
\begin{proof}
In the setting of Euclidean space, the proof was given in \cite[Corollary 3.18]{HRT}. For the case in homogeneous groups, it suffices to note that a similar sharp reverse H\"{o}lder inequality also holds (see \cite{HPR}).
\end{proof}

\subsection{Calder\'{o}n-Zygmund operators with Dini-continuous kernel}
Let $T$ be a bounded linear operator on $L^{2}(\HH)$ represented as
\begin{align*}
T(f)(x)=\int_{\HH}K(x,y)f(y)dy,\ \ \forall x\notin\supp f.
\end{align*}
A function $\omega:[0,1]\rightarrow [0,\infty)$ is a modulus of continuity if it satisfies the following three properties:

(1) $\omega(0)=0$;

(2) $\omega(s)$ is a increasing function;

(3) For any $s_{1},s_{2}>0$, $\omega(s_{1}+s_{2})\leq\omega(s_{1})+\omega(s_{2})$.
\begin{definition}
We say that the operator $T$ is an $\omega$-Calder\'{o}n-Zygmund operator if the kernel $K$ satisfies the following two condition:

(1) (size condition):
\begin{align*}
|K(x,y)|\leq \frac{C_{T}}{d(x,y)^{\mathbb{Q}}},
\end{align*}
for some constant $C_{T}>0$;

(2) (smoothness condition):
\begin{align*}
|K(x,y)-K(x^{\prime},y)|+|K(y,x)-K(y,x^{\prime})|\leq \omega\left(\frac{d(x,x^{\prime})}{d(x,y)}\right)\frac{1}{d(x,y)^{\mathbb{Q}}}
\end{align*}
for $d(x,y)\geq 2A_{0}d(x,x^{\prime})>0$.
\end{definition}

Moreover, $K$ is said to be a \textit{Dini-continuous kernel} if $\omega$ satisfies the \textit{Dini condition}:
\begin{align*}
\|\omega\|_{{\rm Dini}}:=\int_{0}^{1}\omega(s)\frac{ds}{s}<\infty.
\end{align*}
\begin{lemma}\label{domin}
Let $T$ be an $\omega$-Calder\'{o}n-Zygmund operator with $\omega$ satisfying the Dini condition. Then for any $p\in(1,\infty)$ and $w\in A_{p}$, we have
\begin{align}
\|T(f)\|_{L^{p}(w)}\leq C_{\mathbb{Q},p}(\|T\|_{L^{2}\rightarrow L^{2}}+C_{T}+\|\omega\|_{{\rm Dini}})[w]_{A_{p}}^{\max\{\frac{1}{p-1},1\}}\|f\|_{L^{p}(w)}.
\end{align}
\end{lemma}
\begin{proof}
As a special case of \cite[Theorem 6.1]{Lorist},
\begin{align}
\|T(f)\|_{L^{p}(w)}\leq C_{\mathbb{Q},p}(\|T\|_{L^{1}\rightarrow L^{1,\infty}}+\|\omega\|_{{\rm Dini}})[w]_{A_{p}}^{\max\{\frac{1}{p-1},1\}}\|f\|_{L^{p}(w)}.
\end{align}
This, in combination with using a standard Calder\'{o}n-Zygmund decomposition argument to estimate the weak-type norm $\|T\|_{L^{1}\rightarrow L^{1,\infty}}$, completes the proof of Lemma \ref{domin}.
\end{proof}

\subsection{Notation of the paper}
For $1\leq p \leq+\infty$, we denote the norm of a function $f\in L^{p}(\HH)$ by $\|f\|_{p}$. If $T$ is a bounded linear operator
from $L^{p}(\HH)$ to $L^{q}(\HH)$, $1\leq p,q\leq+\infty$, we write $\|T\|_{p\rightarrow q}$ for the operator norm of $T$. The
indicator function of a subset $E\subseteq X$ is denoted by $\chi_{E}$. We use $A\lesssim B$ to denote the statement that $A\leq CB$ for some constant $C>0$.

\bigskip

\bigskip
\section{proof of Theorem \ref{main}}\label{keey}
\setcounter{equation}{0}

In this section, we will combine the ideas from \cite{HRT} and \cite{Tao} to show Theorem \ref{main}.
Throughout this section, we assume that $T$ is a rough homogeneous singular integral operator with $K_{0}$ satisfying the conditions in Theorem \ref{main}.
To show Theorem \ref{main}, the main difficulties
are the lack of suitable Fourier transforms and the convolution on homogeneous groups are not commutative in general. We combine the ideas from \cite{HRT} and \cite{Tao} via using Littlewood--Paley decompositions and Cotlar--Knapp--Stein lemma to overcome the difficulties.

To begin with, we recall that
for appropriate functions $f$ and $g$ on $\HH$, the convolution $f\ast g$ is defined by
\begin{align*}
f\ast g(x)=\int_{\HH}f(y)g(y^{-1}x)dy.
\end{align*}

\subsection{Kernel truncation and frequency localization}
To begin with, we first partion the kernel $K$ dyadically. Note that
\begin{align*}
K=\frac{1}{\ln2}\int_{0}^{\infty}\Delta[t]K_{0}\frac{dt}{t},
\end{align*}
where for each $t$, we define the scaling map $\Delta[t]$ by $$\Delta[t]f(x):=t^{-\mathbb{Q}}f(t^{-1}\circ x).$$
Therefore we have the decomposition
\begin{align*}
K=\sum\limits_{j\in\mathbb{Z}}A_{j}K_{0},
\end{align*}
where $A_{j}$ is the operator
\begin{align}\label{aj}
A_{j}F=2^{-j}\int_{0}^{\infty}\varphi(2^{-j}t)\Delta[t]Fdt
\end{align}
and $\varphi$ is a bump function localized in $\{t\sim 1\}$ such that
$\sum\limits_{j\in\mathbb{Z}}2^{-j}t\varphi(2^{-j}t)=\frac{1}{\ln 2}$.
Hence,
\begin{align*}
T(f)=\sum\limits_{j\in\mathbb{Z}}T_{j}(f),
\end{align*}
where we denote $T_{j}(f)=f\ast A_{j}K_{0}$. Besides, since $\supp K_{0}\subset \{x\in\HH: 1\leq \rho(x)\leq 2\}$, we have for any $q>1$,
\begin{align}\label{Aj}
\|A_{j}K_{0}\|_{1}\leq C\|K_{0}\|_{1}\leq C\|K_{0}\|_{q}.
\end{align}

The next step is to introduce a form of Littlewood-Paley theory, but we avoid any explicit use of the Fourier transform. To this end, let $\phi\in C_{c}^{\infty}(\mathbb{H})$ be a smooth cut-off function such that

(1) $\supp\phi\subset\big\{x\in\HH:\frac{1}{200}\leq\rho(x)\leq\frac{1}{100}\big\}$; (2) $\int_{\HH}\phi(x)dx=1$; (3)$\phi\geq 0$; (4) $\phi=\tilde{\phi}$.
Here $\tilde{F}$ denotes the function $\tilde{F}(x)=F(x^{-1})$. For each integer $j$, write
\begin{align*}
\Psi_{j}=\Delta[2^{j-1}]\phi-\Delta[2^{j}]\phi.
\end{align*}
Then $\Psi_{j}$ is supported on the ball of radius $C2^{j}$, has mean zero, and $\tilde{\Psi}_{j}=\Psi_{j}$.

Next we define the partial sum operators $S_{j}$ by
\begin{align*}
S_{j}(f)=f\ast \Delta[2^{j-1}]\phi.
\end{align*}
Their differences are given by
\begin{align*}
S_{j}(f)-S_{j+1}(f)=f\ast \Psi_{j}.
\end{align*}

Since $S_{j}(f)\rightarrow f$ as $j\rightarrow -\infty$, for any sequence of integer numbers $\{N(j)\}_{j=0}^{\infty}$, with $0=N(0)<N(1)<\cdots<N(j)\rightarrow +\infty$, we have the following identity
\begin{align*}
T_{k}=T_{k}S_{k}+\sum_{j=1}^{\infty}T_{k}(S_{k-N(j)}-S_{k-N(j-1)}).
\end{align*}
In this way, $T=\sum_{j=0}^{\infty}\tilde{T}_{j}=\sum_{j=0}^{\infty}\tilde{T}_{j}^{N}$, where
\begin{align*}
\tilde{T}_{0}:=\tilde{T}_{0}^{N}:=\sum_{k\in\mathbb{Z}}T_{k}S_{k}
\end{align*}
and, for $j\geq 1$,
\begin{align}\label{b1}
\tilde{T}_{j}:=\sum_{k\in\mathbb{Z}}T_{k}(S_{k-j}-S_{k-(j-1)}),
\end{align}
\begin{align}\label{b2}
\tilde{T}_{j}^{N}:=\sum_{k\in\mathbb{Z}}T_{k}(S_{k-N(j)}-S_{k-N(j-1)})=\sum_{i=N(j-1)+1}^{N(j)}\tilde{T}_{i}.
\end{align}
For $j\in\mathbb{Z}$, we also consider the operator $G_{j}$ defined by
\begin{align}\label{b3}
G_{j}=\sum_{k\in\mathbb{Z}}T_{k}(S_{k-j}-S_{k-(j-1)}).
\end{align}
Then we can further decompose $\tilde{T}_{0}$ in the following way.
\begin{align}\label{equa}
\tilde{T}_{0}=\sum_{j=-\infty}^{0}G_{j}.
\end{align}
\subsection{$L^{2}$ estimate for $\tilde{T}_{j}^{N}$}
In this subsection, we will give the $L^{2}$-estimate of the operators $\tilde{T}_{j}^{N}$, which plays a crucial role in the proof of Theorem \ref{main}.
\begin{proposition}\label{ll1}
Let $q>1$ and $\tilde{T}_{j}$ and $\tilde{T}_{j}^{N}$ be the operators as in \eqref{b1} and \eqref{b2}. Then there exist constants $C_{\mathbb{Q},q}>0$ and $\alpha>0$ such that for any $j\geq 0$,
\begin{align}\label{ggg1}
\|\tilde{T}_{j}(f)\|_{2}\leq C_{\mathbb{Q},q}2^{-\alpha j}\|K_{0}\|_{q}\|f\|_{2}
\end{align}
and for any $j\geq 1$,
\begin{align}\label{ggggg1}
\|\tilde{T}_{j}^{N}(f)\|_{2}\leq C_{\mathbb{Q},q}2^{-\alpha N(j-1)}\|K_{0}\|_{q}\|f\|_{2}.
\end{align}
\end{proposition}

Now, we show the unweighted $L^{2}$ estimate for $G_{j}$ and then $\tilde{T}_{j}^{N}$.
\begin{lemma}\label{ll2}
Let $q>1$ and $G_{j}$ be the operator as in \eqref{b3}. Then there exist constants $C_{\mathbb{Q},q}>0$ and $\alpha>0$ such that for any $j\in\mathbb{Z}$,
\begin{align*}
\|G_{j}(f)\|_{2}\leq C_{\mathbb{Q},q}2^{-\alpha |j|}\|K_{0}\|_{q}\|f\|_{2}.
\end{align*}
\end{lemma}
\begin{proof}
For simplicity, we set
$$G_{j,k}(f):=T_{k}(S_{k-j}-S_{k-(j-1)})(f),$$
then $G_{j}(f)=\sum_{k\in\mathbb{Z}}G_{j,k}(f)$. By Cotlar-Knapp-Stein Lemma (see \cite{s93}), it suffices to show that:
\begin{align}\label{CKS}
\|G_{j,k}^{*}G_{j,k^{\prime}}\|_{2\rightarrow 2}+\|G_{j,k^{\prime}}G_{j,k}^{*}\|_{2\rightarrow 2}\leq C2^{-2\alpha |j|}2^{-c|k-k^{\prime}|}\|K_{0}\|_{q}^{2},
\end{align}
for some $C,c>0$ and $\alpha>0$. We only estimate the first term, since the second term is similar and simpler. A direct calculation yields
\begin{align}\label{GDE}
G_{j,k}^{*}G_{j,k^{\prime}}(f)
&=f\ast \Psi_{k^{\prime}-j}\ast A_{k^{\prime}}K_{0}\ast A_{k}\tilde{K}_{0}\ast \Psi_{k-j}\nonumber\\
&=\sum_{\ell=1}^{\infty}\sum_{\ell^{\prime}=1}^{\infty}f\ast \Psi_{k^{\prime}-j}\ast A_{k^{\prime}}K_{0}\ast\Psi_{\ell}\ast\Psi_{\ell^{\prime}}\ast A_{k}\tilde{K}_{0}\ast \Psi_{k-j}.
\end{align}

To continue, we recall that Tao \cite{Tao} applied iterated $(TT^{*})^{N}$ method to obtain the following inequality with $q=\infty$ and then Sato \cite{sato} extended it to general $q>1$:
there exist constants $C_{\mathbb{Q},q}>0$ and $\alpha>0$ such that for any integers $j$, $k$, and any $L^{q}$ function $F$ on the annulus with mean zero,
\begin{align}\label{kkkey}
\|f\ast A_{j}F\ast \Psi_{k}\|_{2}\leq C_{\mathbb{Q},q}2^{-\alpha|j-k|}\|f\|_{2}\|F\|_{q}.
\end{align}

On the one hand, it follows from the above fact and its duality version that
\begin{align}\label{f1}
\|(f\ast \Psi_{k^{\prime}-j})\ast (A_{k^{\prime}}K_{0}\ast\Psi_{\ell})\ast(\Psi_{\ell^{\prime}}\ast A_{k}\tilde{K}_{0})\ast \Psi_{k-j}\|_{2}
&\leq C\|(f\ast \Psi_{k^{\prime}-j})\ast (A_{k^{\prime}}K_{0}\ast\Psi_{\ell})\ast(\Psi_{\ell^{\prime}}\ast A_{k}\tilde{K}_{0})\|_{2}\nonumber\\
&\leq C2^{-\beta|k-\ell^{\prime}|}\|K_{0}\|_{q}\|(f\ast \Psi_{k^{\prime}-j})\ast (A_{k^{\prime}}K_{0}\ast\Psi_{\ell})\|_{2}\nonumber\\
&\leq C2^{-\beta|k-\ell^{\prime}|}2^{-\beta|k^{\prime}-\ell|}\|K_{0}\|_{q}^{2}\|f\ast \Psi_{k^{\prime}-j}\|_{2}\nonumber\\
&\leq C2^{-\beta|k-\ell^{\prime}|}2^{-\beta|k^{\prime}-\ell|}\|K_{0}\|_{q}^{2}\|f\|_{2}
\end{align}
for some $C>0$ and $\beta>0$.

On the other hand, by the smoothness and the cancellation properties of $\Psi_{\ell}$ and $\Psi_{\ell^{\prime}}$, we see that $\|\Psi_{\ell}\ast\Psi_{\ell^{\prime}}\|_{1}\leq C2^{-|\ell-\ell^{\prime}|}$. Therefore,
\begin{align}\label{vbnm}
\|(f\ast \Psi_{k^{\prime}-j}\ast A_{k^{\prime}}K_{0})\ast(\Psi_{\ell}\ast\Psi_{\ell^{\prime}})\ast (A_{k}\tilde{K}_{0}\ast \Psi_{k-j})\|_{2}
&\leq C2^{-\beta|j|}\|K_{0}\|_{q}\|(f\ast \Psi_{k^{\prime}-j}\ast A_{k^{\prime}}K_{0})\ast(\Psi_{\ell}\ast\Psi_{\ell^{\prime}})\|_{2}\nonumber\\
&\leq C2^{-\beta|j|}2^{-|\ell-\ell^{\prime}|}\|K_{0}\|_{q}\|f\ast \Psi_{k^{\prime}-j}\ast A_{k^{\prime}}K_{0}\|_{2}\nonumber\\
&\leq C2^{-2\beta|j|}2^{-|\ell-\ell^{\prime}|}\|K_{0}\|_{q}^{2}\|f\|_{2}.
\end{align}

Taking geometric means of \eqref{f1} and \eqref{vbnm} and then applying the triangle inequality, we see that
\begin{align*}
\|(f\ast \Psi_{k^{\prime}-j}\ast A_{k^{\prime}}K_{0})\ast(\Psi_{\ell}\ast\Psi_{\ell^{\prime}})\ast (A_{k}\tilde{K}_{0}\ast \Psi_{k-j})\|_{2}\lesssim C2^{-c|k-\ell^{\prime}|}2^{-c|k^{\prime}-\ell|}2^{-c|j|}2^{-c|k-k^{\prime}|}\|K_{0}\|_{q}^{2}\|f\|_{2},
\end{align*}
for some constant $c>0$.
This, in combination with the equality \eqref{GDE}, implies the estimate \eqref{CKS}. This finishes the proof of Lemma \ref{ll2}.
\end{proof}

Next, we give the proof of Proposition \ref{ll1}.

{\it Proof of Proposition \ref{ll1}}.
It suffices to show the first inequality, since the second one can be obtained by simply summing the geometric series $\sum_{i=N(j-1)+1}^{N(j)}2^{-\alpha i}$. It follows from the fact $G_{j}=\tilde{T}_{j}$ for $j\geq 1$ and Lemma \ref{ll2} that the estimate \eqref{ggg1} holds for $j\geq 1$. For the case $j=0$, By applying Lemma \ref{ll2} to the equality \eqref{equa}, we obtain
\begin{align*}
\|\tilde{T}_{0}(f)\|_{2}\leq \sum_{j=-\infty}^{0}\|G_{j}(f)\|_{2}\leq C\sum_{j=-\infty}^{0}2^{\alpha j}\|K_{0}\|_{q}\|f\|_{2}\leq C\|K_{0}\|_{q}\|f\|_{2}.
\end{align*}
Hence, the proof of Proposition \ref{ll1} is finished.
\hfill $\square$

\subsection{Calder\'{o}n--Zygmund theory of $\tilde{T}_{j}^{N}$}
\begin{lemma}\label{cal}
The operator $\tilde{T}_{j}^{N}$ is a Calder\'{o}n--Zygmund operator satisfying for any $q>1$,
\begin{align*}
C_{j}^{N}:=C_{\tilde{T}_{j}^{N}}\leq C_{\mathbb{Q},q}2^{\frac{N(j)\mathbb{Q}}{q}}\|K_{0}\|_{q},\ \ \omega_{j}^{N}(t):=\omega_{\tilde{T}_{j}^{N}}(t)\leq C_{\mathbb{Q},q}2^{\frac{N(j)\mathbb{Q}}{q}}\min\{1,2^{N(j)}t\}\|K_{0}\|_{q},
\end{align*}
which satisfies
\begin{align*}
\int_{0}^{1}\omega_{j}^{N}(t)\frac{dt}{t}\leq C_{\mathbb{Q},q}2^{\frac{N(j)\mathbb{Q}}{q}}(1+N(j))\|K_{0}\|_{q}.
\end{align*}
\end{lemma}
\begin{proof}
From Proposition \ref{ll1} we can see that $\tilde{T}_{j}^{N}$ is a bounded operator in $L^{2}$. In order to obtain the required estimates for the kernel of $\tilde{T}_{j}^{N}$, we first study the kernel of each $T_{k}S_{k-N(j)}$. Note that
\begin{align*}
|A_{k}K_{0}(x)|
&=2^{-k}\left|\int_{-\infty}^{+\infty}\varphi(2^{-k}t)t^{-\mathbb{Q}}K(t^{-1}\circ x)\chi_{1\leq \rho(t^{-1}\circ x)\leq 2}(x)dt\right|\\
&=\left|\int_{-\infty}^{+\infty}t\varphi(t)(2^{k}t)^{-\mathbb{Q}}K((2^{k}t)^{-1}\circ x)\chi_{2^{k}t\leq\rho(x)\leq 2^{k+1}t}(x)dt\right|\\
&\leq C\rho(x)^{-\mathbb{Q}}|K(\rho(x)^{-1}\circ x)|\chi_{2^{k-1}\leq\rho(x)\leq 2^{k+2}}(x).
\end{align*}
Since $\supp\phi\subset\{x\in\HH:\rho(x)\leq \frac{1}{100}\}$,
\begin{align*}
&|\Delta[2^{k-N(j)-1}]\phi\ast A_{k}K_{0}(x)|\\
&\leq C\int_{\HH}2^{-[k-N(j)-1]\mathbb{Q}}|\phi(2^{-[k-N(j)-1]}\circ xy^{-1})|\rho(y)^{-\mathbb{Q}}\chi_{2^{k-1}\leq\rho(y)\leq 2^{k+2}}(y)|K(\rho(y)^{-1}\circ y)|dy\\
&\leq C2^{-[k-N(j)-1]\mathbb{Q}}\rho(x)^{-\mathbb{Q}}\chi_{2^{k-2}\leq \rho(x)\leq 2^{k+3}}(x)\int_{2^{k-1}\leq \rho(y)\leq 2^{k+2}}|\phi(2^{-[k-N(j)-1]}\circ xy^{-1})||K(\rho(y)^{-1}\circ y)|dy.
\end{align*}
By H\"{o}lder's inequality,
\begin{align*}
&\int_{2^{k-1}\leq \rho(y)\leq 2^{k+2}}|\phi(2^{-[k-N(j)-1]}\circ xy^{-1})||K(\rho(y)^{-1}\circ y)|dy\\
&\leq \left(\int_{2^{k-1}\leq \rho(y)\leq 2^{k+2}}|K(\rho(y)^{-1}\circ y)|^{q}dy\right)^{1/q}\left(\int_{\HH}|\phi(2^{-[k-N(j)-1]}\circ xy^{-1})|^{q^{\prime}}dy\right)^{1/q^{\prime}}\\
&\leq C2^{\frac{k\mathbb{Q}}{q}}2^{\frac{(k-N(j)-1)\mathbb{Q}}{q^{\prime}}}\|K_{0}\|_{q}.
\end{align*}
Therefore,
\begin{align}\label{p1}
\sum_{k\in\mathbb{Z}}|\Delta[2^{k-N(j)-1}]\phi\ast A_{k}K_{0}(x)|
\leq C2^{\frac{N(j)\mathbb{Q}}{q}}\rho(x)^{-\mathbb{Q}}\|K_{0}\|_{q}.
\end{align}
Similarly, we obtain the  gradient estimate as follows.
\begin{align*}
&|\nabla\Delta[2^{k-N(j)-1}]\phi\ast A_{k}K_{0}(x)|\leq C2^{N(j)\big(1+\frac{\mathbb{Q}}{q}\big)}\rho(x)^{-\mathbb{Q}-1}\chi_{2^{k-2}\leq \rho(x)\leq 2^{k+3}}(x)\|K_{0}\|_{q}.
\end{align*}
Hence,
\begin{align}\label{p2}
\sum_{k\in\mathbb{Z}}|\nabla\Delta[2^{k-N(j)-1}]\phi\ast A_{k}K_{0}(x)|\leq C2^{N(j)\big(1+\frac{\mathbb{Q}}{q}\big)}\rho(x)^{-\mathbb{Q}-1}\|K_{0}\|_{q}.
\end{align}
From the triangle inequality and $N(j-1)<N(j)$ we can see that the kernel $K_{j}^{N}:=\sum_{k\in\mathbb{Z}}(\Delta[2^{k-N(j)-1}]\phi-\Delta[2^{k-N(j-1)-1}]\phi)\ast A_{k}K_{0}$ of $\tilde{T}_{j}^{N}$ satisfies the same estimates \eqref{p1} and \eqref{p2}. That is,
\begin{align*}
|K_{j}^{N}(x,y)|=|K_{j}^{N}(y^{-1}x)|\leq C2^{\frac{N(j)\mathbb{Q}}{q}}d(x,y)^{-\mathbb{Q}}\|K_{0}\|_{q},\\
|\nabla K_{j}^{N}(x,y)|\leq C2^{N(j)\big(1+\frac{\mathbb{Q}}{q}\big)}d(x,y)^{-\mathbb{Q}-1}\|K_{0}\|_{q}.
\end{align*}
(For $j=0$, the subtraction is not even needed.) The first bound above is already the required estimate for $C_{j}^{N}$. Besides, for $d(x,y)\geq 2A_{0}d(x,x^{\prime})$, by mean value theorem on homogeneous groups (see for example \cite{FoSt}),
\begin{align*}
|K_{j}^{N}(x,y)-K_{j}^{N}(x^{\prime},y)|
&=|K_{j}^{N}(y^{-1}x)-K_{j}^{N}(y^{-1}x^{\prime})|\\
&\leq C2^{N(j)\big(1+\frac{\mathbb{Q}}{q}\big)}d(x,y)^{-\mathbb{Q}-1}d(x,x^{\prime})\|K_{0}\|_{q}.
\end{align*}
By the triangle inequality, we also have
\begin{align*}
|K_{j}^{N}(x,y)-K_{j}^{N}(x^{\prime},y)|\leq C2^{\frac{N(j)\mathbb{Q}}{q}}d(x,y)^{-\mathbb{Q}}\|K_{0}\|_{q}.
\end{align*}
Combining the two estimates we obtain above and by symmetry, we conclude that
\begin{align*}
|K_{j}^{N}(x,y)-K_{j}^{N}(x^{\prime},y)|+|K_{j}^{N}(y,x)-K_{j}^{N}(y,x^{\prime})|\leq C\omega_{j}^{N}\left(\frac{d(x,x^{\prime})}{d(x,y)}\right)d(x,y)^{-\mathbb{Q}},
\end{align*}
where
\begin{align*}
\omega_{j}^{N}(t)\leq C2^{\frac{N(j)\mathbb{Q}}{q}}\min\{1,2^{N(j)}t\}\|K_{0}\|_{q}.
\end{align*}
A direct calculation yields that
\begin{align*}
\int_{0}^{1}\omega_{j}^{N}(t)\frac{dt}{t}\leq C2^{\frac{N(j)\mathbb{Q}}{q}}(1+N(j))\|K_{0}\|_{q}.
\end{align*}
This ends the proof of Lemma \ref{cal}.
\end{proof}
With the help of  Proposition \ref{ll1}, Lemmata \ref{domin} and \ref{cal}, we can easily follow a similar procedure in \cite{HRT} to show a bad quantitative $L^{p}$ weighted inequality and a good quantitative unweighted $L^{p}$ estimate for the operators $\tilde{T}_{j}^{N}$.
\begin{lemma}\label{wei}
Let $1<p<\infty$ and $q>1$, then for any $w\in A_{p}$, there exists a constant $C_{\mathbb{Q},p,q}>0$ such that
\begin{align*}
\|\tilde{T}_{j}^{N}(f)\|_{L^{p}(w)}\leq C_{\mathbb{Q},p,q}2^{\frac{N(j)\mathbb{Q}}{q}}(1+N(j))\|K_{0}\|_{q}\{w\}_{A_{p}}\|f\|_{L^{p}(w)}.
\end{align*}
\end{lemma}
\begin{proof}
By Proposition \ref{ll1}, Lemmata \ref{domin} and \ref{cal},
\begin{align*}
\|\tilde{T}_{j}^{N}\|_{L^{p}(w)}
&\leq C_{\mathbb{Q},p}(\|\tilde{T}_{j}^{N}\|_{L^{2}\rightarrow L^{2}}+C_{j}^{N}+\|\omega_{j}^{N}\|_{{\rm Dini}})\{w\}_{A_{p}}\|f\|_{L^{p}(w)}\\
&\leq C_{\mathbb{Q},p,q}(2^{-\alpha N(j-1)}\|K_{0}\|_{q}+\|K_{0}\|_{q}+2^{\frac{N(j)\mathbb{Q}}{q}}(1+N(j))\|K_{0}\|_{q})\{w\}_{A_{p}}\|f\|_{L^{p}(w)}\\
&\leq C_{\mathbb{Q},p,q}2^{\frac{N(j)\mathbb{Q}}{q}}(1+N(j))\|K_{0}\|_{q}\{w\}_{A_{p}}\|f\|_{L^{p}(w)}.
\end{align*}
This finishes the proof of Lemma \ref{wei}.
\end{proof}
\begin{lemma}\label{unwei}
Let $1<p<\infty$ and $q>1$, then there exist constants $C_{\mathbb{Q},p,q}$ and $\beta_{p}>0$ such that
\begin{align*}
\|\tilde{T}_{j}^{N}(f)\|_{L^{p}}\leq C_{\mathbb{Q},p,q}2^{-\beta_{p}N(j-1)}2^{\frac{N(j)\mathbb{Q}}{q}}(1+N(j))\|K_{0}\|_{q}\|f\|_{L^{p}}.
\end{align*}
\end{lemma}
\begin{proof}
We first consider the case $p>2$ and let $s=2p$ so that $2<p<s$.  This implies that $\frac{1}{p}=\frac{1-\theta}{2}+\frac{\theta}{s}$, for $0<\theta:=\frac{p-2}{p-1}<1$. Then, it follows from Proposition \ref{ll1}, Lemma \ref{wei} with $w(x)\equiv 1$ and complex interpolation that
\begin{align*}
\|\tilde{T}_{j}^{N}\|_{L^{p}\rightarrow L^{p}}
&\leq \|\tilde{T}_{j}^{N}\|_{L^{2}\rightarrow L^{2}}^{1-\theta}\|\tilde{T}_{j}^{N}\|_{L^{2p}\rightarrow L^{2p}}^{\theta}\\
&\leq (C_{\mathbb{Q},q}2^{-\alpha N(j-1)}\|K_{0}\|_{q})^{1-\theta}(C_{\mathbb{Q},2p,q}2^{\frac{N(j)\mathbb{Q}}{q}}(1+N(j))\|K_{0}\|_{q})^{\theta}\\
&\leq C_{\mathbb{Q},p,q}2^{-\beta_{p} N(j-1)}2^{\frac{N(j)\mathbb{Q}}{q}}(1+N(j))\|K_{0}\|_{q},
\end{align*}
where $\beta_{p}=\alpha(1-\theta)=\alpha/(p-1)$.

For the case $p<2$, let $s:=\frac{2p}{1+p}$  so that $1<s<p<2$. In this case, $\frac{1}{p}=\frac{1-\theta}{2}+\frac{\theta}{s}$, for $0<\theta:=2-p<1$. Applying the interpolation theorem between $L^{2}$ and $L^{s}$, we obtain a similar $L^{p}$ estimate.
\end{proof}
\subsection{Proof of Theorem \ref{main}}
Let us denote $\varepsilon:=\frac{1}{2}c_{\mathbb{Q}}/(w)_{A_{p}}$. It follows from Lemmata \ref{wei} and \ref{wei2} that for this choice of $\varepsilon$,
\begin{align*}
\|\tilde{T}_{j}^{N}\|_{L^{p}(w^{1+\varepsilon})\rightarrow L^{p}(w^{1+\varepsilon})}
&\leq C_{\mathbb{Q},p}2^{\frac{N(j)\mathbb{Q}}{q}}(1+N(j))\|K_{0}\|_{q}\{w^{1+\varepsilon}\}_{A_{p}}\\
&\leq C_{\mathbb{Q},p}2^{\frac{N(j)\mathbb{Q}}{q}}(1+N(j))\|K_{0}\|_{q}\{w\}_{A_{p}}^{1+\varepsilon}.
\end{align*}
Besides, by Lemma \ref{unwei}, we also have
\begin{align*}
\|\tilde{T}_{j}^{N}\|_{L^{p}\rightarrow L^{p}}\leq C_{\mathbb{Q},p}2^{-\beta_{p}N(j-1)}2^{\frac{N(j)\mathbb{Q}}{q}}(1+N(j))\|K_{0}\|_{L^{q}}.
\end{align*}
We now apply the interpolation theorem with change of measures (\cite[Theorem 2.11]{SW}) to $T=\tilde{T}_{j}^{N}$ with $p_{0}=p_{1}=p$, $w_{0}=1$ and $w_{1}=w^{1+\varepsilon}$ so that $\theta=\varepsilon/(1+\varepsilon)$ and
\begin{align*}
\|\tilde{T}_{j}^{N}\|_{L^{p}(w)\rightarrow L^{p}(w)}
&\leq \|\tilde{T}_{j}^{N}\|_{L^{p}\rightarrow L^{p}}^{\varepsilon/(1+\varepsilon)}\|\tilde{T}_{j}^{N}\|_{L^{p}(w^{1+\varepsilon})\rightarrow L^{p}(w^{1+\varepsilon})}^{1/(1+\varepsilon)}\\
&\leq C_{\mathbb{Q},p,q}\|K_{0}\|_{q}2^{\frac{N(j)\mathbb{Q}}{q}}2^{-\beta_{p}N(j-1)\varepsilon/(1+\varepsilon)}(1+N(j))\{w\}_{A_{p}}\\
&\leq C_{\mathbb{Q},p,q}\|K_{0}\|_{q}2^{\frac{N(j)\mathbb{Q}}{q}}2^{-\beta_{\mathbb{Q},p}N(j-1)/(w)_{A_{p}}}(1+N(j))\{w\}_{A_{p}},
\end{align*}
for some constants $\beta_{p},\beta_{\mathbb{Q},p}>0$.

Thus,
\begin{align}\label{last}
\|T\|_{L^{p}(w)\rightarrow L^{p}(w)}&\leq \sum_{j=0}^{\infty}\|\tilde{T}_{j}^{N}\|_{L^{p}(w)\rightarrow L^{p}(w)}\nonumber\\
&\leq C_{\mathbb{Q},p,q}\|K_{0}\|_{q}\{w\}_{A_{p}}\sum_{j=0}^{\infty}2^{\frac{N(j)\mathbb{Q}}{q}}2^{-\beta_{\mathbb{Q},p}N(j-1)/(w)_{A_{p}}}(1+N(j)).
\end{align}
Note that if we choose $N(j)=2^{j}$ for $j\geq 1$ and $q>\frac{2\mathbb{Q}(w)_{A_{p}}}{\beta_{\mathbb{Q},p}}:=c_{\mathbb{Q},p}(w)_{A_{p}}$, then
\begin{align*}
\sum_{j=0}^{\infty}2^{\frac{N(j)\mathbb{Q}}{q}}2^{-\beta_{\mathbb{Q},p}N(j-1)/(w)_{A_{p}}}(1+N(j))
&\leq \sum_{j=0}^{\infty}2^{j}2^{-\beta_{\mathbb{Q},p,q}2^{j}(w)_{A_{p}}^{-1}}\\
&\leq C_{\mathbb{Q},p,q}\left(\sum_{j:2^{j}\leq (w)_{A_{p}}}2^{j}+\sum_{j:2^{j}>(w)_{A_{p}}}2^{j}\left(\frac{(w)_{A_{p}}}{2^{j}}\right)^{2}\right)\\
&\leq C_{\mathbb{Q},p,q}(w)_{A_{p}},
\end{align*}
for some constant $\beta_{\mathbb{Q},p,q}>0$.

This, in combination with the estimate \eqref{last}, completes the proof of Theorem \ref{main}.
\hfill $\square$

\section{Application: quantitative estimate of singular integrals studied by Sato}\label{satosec}
\setcounter{equation}{0}

In the previous section, we proved quantitative weighted estimates for classical rough homogeneous singular integrals on homogeneous group. Indeed, our argument can also be applied to draw a parallel conclusion for a larger class of singular integrals considered by \cite{sato}.

\begin{proof}[Proof of Theorem \ref{main2}]
The proof is a minor modification of Theorem \ref{main}. We list the differences of the proof here. To begin with, we decompose the kernel $L$ into the summation of $B_{j}K_{0}$,
where $B_{j}$ is the operator
\begin{align*}
B_{j}F=2^{-j}h(\rho(x))\int_{0}^{\infty}\varphi(2^{-j}t)\Delta[t]Fdt.
\end{align*}
Then the operators $T_{j}$, $\tilde{T}_{j}$, $\tilde{T}_{j}^{N}$ and  $G_{j}$, $G_{j,k}$ can be constructed with $A_{j}$ (defined by \eqref{aj}) replaced by $B_{j}$.
We note that
\begin{align*}
&G_{j,k}^{*}G_{j,k^{\prime}}(f)=\sum_{\ell\in\mathbb{Z}}\sum_{\ell^{\prime}\in\mathbb{Z}}f\ast \Psi_{k^{\prime}-j}\ast B_{k^{\prime}}K_{0}\ast \Psi_{\ell}\ast \Psi_{\ell^{\prime}}\ast \widetilde{B_{k}K_{0}}\ast \Psi_{k-j}.
\end{align*}
This, combined with  Lemma 1 in \cite{sato} and Young's inequality, yields that
\begin{align}\label{g1}
\|G_{j,k}^{*}G_{j,k^{\prime}}(f)\|_{2}
&\leq C_{\mathbb{Q},q}2^{-2\alpha |j|}2^{-c|k-k^{\prime}|}\|K_{0}\|_{q}^{2}\|h\|_{\Lambda_{q}^{\eta/q^{\prime}}}^{2}\|f\|_{2}.
\end{align}
Similar to the proof of Proposition \ref{ll1}, we obtain that for any $j\geq 1$,
\begin{align}\label{ggggg22}
\|\tilde{T}_{j}^{N}(f)\|_{2}\leq C_{\mathbb{Q},q}2^{-\alpha N(j-1)}\|K_{0}\|_{q}\|h\|_{\Lambda_{q}^{\eta/q^{\prime}}}\|f\|_{2}.
\end{align}

In the next step, note that for $2/q+1/q_{0}=1$,
\begin{align*}
&\int_{2^{k-1}\leq \rho(y)\leq 2^{k+2}}|\phi(2^{-[k-N(j)-1]}\circ xy^{-1})||K(\rho(y)^{-1}\circ y)||h(\rho(y))|dy\\
&\leq \left(\int_{2^{k-1}\leq \rho(y)\leq 2^{k+2}}|K(\rho(y)^{-1}\circ y)|^{q}dy\right)^{1/q}\left(\int_{2^{k-1}\leq \rho(y)\leq 2^{k+2}}|h(\rho(y))|^{q}dy\right)^{1/q}\left(\int_{\HH}|\phi(2^{-[k-N(j)-1]}\circ y)|^{q_{0}}dy\right)^{1/q_{0}}\\
&\leq C2^{kQ}2^{-\frac{N(j)Q}{q_{0}}}\|K_{0}\|_{q}\|h\|_{d_{q}}.
\end{align*}

Then a simple modification of Lemma \ref{cal} yields that the operator $\tilde{T}_{j}^{N}$ is a Calder\'{o}n--Zygmund operator satisfying for any $q>2$,
\begin{align*}
C_{j}^{N}:=C_{\tilde{T}_{j}^{N}}\leq C_{\mathbb{Q},q}2^{\frac{2N(j)\mathbb{Q}}{q}}\|K_{0}\|_{q}\|h\|_{d_{q}},\ \ \omega_{j}^{N}(t):=\omega_{\tilde{T}_{j}^{N}}(t)\leq C_{\mathbb{Q},q}2^{\frac{2N(j)\mathbb{Q}}{q}}\min\{1,2^{N(j)}t\}\|K_{0}\|_{q}\|h\|_{d_{q}},
\end{align*}
which satisfies
\begin{align*}
\int_{0}^{1}\omega_{j}^{N}(t)\frac{dt}{t}\leq C_{\mathbb{Q},q}2^{\frac{2N(j)\mathbb{Q}}{q}}(1+N(j))\|K_{0}\|_{q}\|h\|_{d_{q}}.
\end{align*}
Next we obtain simple variants of Lemmata \ref{wei} and \ref{unwei}, and then the proof of Theorem \ref{main2} is complete.
\end{proof}

\section{An investigation in the bi-parameter setting: proof of Theorem \ref{main3}}\label{bi}
\setcounter{equation}{0}

In this section, we show that our argument can be also applied to obtain a parallel result in the bi-parameter setting.
To begin with, let $\HH_{i}=\RR^{n_{i}}$, $i=1,2$, be homogeneous groups with dilations $\circ_{1}$, $\circ_{2}$, and norm functions $\rho_{1}$, $\rho_{2}$, respectively. Each $\circ_{i}$ is an automorphism of the group structure and is of the form
\begin{align*}
t\circ_{i} (x_{i}^{1},\ldots,x_{i}^{n})=(t^{\alpha_{1}^{i}}x_{i}^{1},\ldots,t^{\alpha_{n}^{i}}x_{i}^{n}),\ \forall (x_{i}^{1},\ldots,x_{i}^{n})\in\HH_{i},
\end{align*}
for some constants $0<\alpha_{1}^{i}\leq \alpha_{2}^{i}\leq\ldots\leq \alpha_{n}^{i}$. We call the quantity $\mathbb{Q}_{i}=\sum_{j=1}^{n_{i}}\alpha_{j}^{i}$ the homogeneous dimension of $\HH_{i}$.
We define a left-invariant quasi-distance $d_{i}$ on $\HH_{i}$ by $d_{i}(x,y)=\rho_{i}(x^{-1}y)$, which means that
there exists a constant $A_{0}^{(i)}\geq 1$ such that for any $x,y,z\in\HH_{i}$,
$d_{i}(x,y)\leq A_{0}^{(i)}[d_{i}(x,z)+d_{i}(z,y)].
$ 
 Let $B_{i}(x_{i},r)$ be the ball with center $x_{i}\in\HH_{i}$ and radius $t\in\mathbb{R}_{+}$ defined by
$B_{i}(x_{i},r)=\{y_{i}\in\HH_{i}:d_{i}(x_{i},y_{i})<r\}.
$ 
\begin{definition}
Let $w(x_1,x_2)$ be a nonnegative locally integrable function
  on~$\HH_{1}\times\HH_{2}$. For $1 < p < \infty$, we
  say that $w$ is a product $A_p$ \emph{weight}, written $w\in
  A_p$, if
\begin{align*}
[w]_{A_{p}}:=\sup\limits_{R}\left(\fint_{R}wdx_1dx_2\right)\left(\fint_{R}\left(\frac{1}{w}\right)^{1/(p-1)}dx_1dx_2\right)^{p-1}<\infty,
\end{align*}
where the supremum is taken over all rectangles~$R\subset \HH_{1}\times\HH_{2}$. The quantity $[w]_{A_p}$ is called the \emph{$A_p$~constant
  of~$w$}.
  For $p = 1$, if $M_s(w)(x_1,x_2)\leq w(x_1,x_2)$ for a.e. $(x_1,x_2)\in \HH_{1}\times\HH_{2}$, then we say that $w$ is a product $A_1$ \emph{weight},
  written $w\in A_1$, where $M_s$ denotes the strong maximal function on $\HH_{1}\times\HH_{2}$. Besides, let $A_\infty := \cup_{1\leq p<\infty} A_p$ and we have
\begin{align*}
[w]_{A_{\infty}}:=\sup\limits_{R}\left(\fint_{R}wdx_1dx_2\right){\rm exp}\left(\fint_{R}{\rm log}\left(\frac{1}{w}\right)dx_1dx_2\right)<\infty.
\end{align*}
\end{definition}

Throughout this section, for appropriate functions $f$ and $g$ on $\HH_{1}\times\HH_{2}$, the convolution $f\ast g$ is defined by
\begin{align*}
(f\ast g)(x,y)=\int_{\HH_{1}\times\HH_{2}}f(xu^{-1},yv^{-1})g(u,v)dudv.
\end{align*}





We now provide the proof of Theorem \ref{main3}.

\medskip
\begin{proof}[{\bf Proof of Theorem \ref{main3}}]

It suffices for us to provide the decomposition of $T$ into a suitable collection $\{\tilde T_{j_1,j_2}^N\}$ and then to verify
that each $\tilde  T_{j_1,j_2}^N$ is a paraproduct-free operator in the class of bi-parameter singular integral operators with the modulus of continuity $\omega_{j_{1},j_{2}}^{N,k}$ on $\HH_k$, $k=1,2$, satisfying the modified Dini$_1$ condition \cite[Section 5]{AMV}:
\begin{align}\label{Dini1}
\|\omega_{j_{1},j_{2}}^{N,k}\|_{{\rm Dini}_1}:=\int_0^1 \omega_{j_{1},j_{2}}^{N,k}(t) \Big( 1+ \log {1\over t} \Big){dt\over t}
\leq C2^{\frac{N(j_{1})\mathbb{Q}_{1}}{q}}2^{\frac{N(j_{2})\mathbb Q_{2}}{q}}\|K^{0}\|_{q}(1+N(j_{k}))^{2}.
\end{align}

\medskip
We first partion the kernel $K$ dyadically.
\begin{align*}
K(x,y)=\frac{1}{(\ln2)^{2}}\int_{0}^{\infty}\int_{0}^{\infty}\Delta[t_{1},t_{2}]K^{0}(x,y)\frac{dt_{1}}{t_{1}}\frac{dt_{2}}{t_{2}},
\end{align*}
where for each $t_{1},t_{2}\in\mathbb{R}_{+}$, we define the product scaling map $\Delta[t_{1},t_{2}]$ by $\Delta[t_{1},t_{2}]=\Delta^{(1)}[t_{1}]\otimes \Delta^{(2)}[t_{2}],$
and
$\Delta^{(i)}[t_{i}]f(x_{i}):=t_{i}^{-\mathbb{Q}_{i}}f(t_{i}^{-1}\circ_{i}x_{i})$, $i=1,2$. Therefore we have the decomposition
\begin{align*}
T(f)(x,y)=\sum\limits_{(j_{1},j_{2})\in\mathbb{Z}^{2}}f\ast A_{j_{1},j_{2}}K^{0}(x,y)=:\sum\limits_{(j_{1},j_{2})\in\mathbb{Z}^{2}}T_{j_{1},j_{2}}(f)(x,y),
\end{align*}
where $A_{j_{1},j_{2}}$ is the operator
\begin{align}
A_{j_{1},j_{2}}F(x,y)=2^{-j_{1}}2^{-j_{2}}\int_{0}^{\infty}\int_{0}^{\infty}\varphi(2^{-j_{1}}t_{1})\varphi(2^{-j_{2}}t_{2})\Delta[t_{1},t_{2}]F(x,y)dt_{1}dt_{2}.
\end{align}

Let $\phi^{(i)}\in C_{c}^{\infty}(\mathbb{H}_{i})$ be a smooth cut-off function supported in $B_{i}(0,\frac{1}{100})\backslash B_{i}(0,\frac{1}{200})$ such that $\int_{\HH_{i}}\phi^{(i)}dx_{i}=1$, $\phi^{(i)}=\widetilde{\phi^{(i)}}$, $\phi^{(i)}(x_{i})\geq 0$ for all $x_{i}\in\HH_{i}$.
%
Denote
\begin{align*}
\Psi_{j}^{(i)}=\Delta^{(i)}[2^{j-1}]\phi^{(i)}-\Delta^{(i)}[2^{j}]\phi^{(i)},
\end{align*}
then $\Psi_{j}^{(i)}$ satisfies $\supp\Psi_{j}^{(i)}\subseteq B_{i}(0,C2^{j})$, has mean zero, and $\widetilde{\Psi_{j}^{(i)}}=\Psi_{j}^{(i)}$.

We define the partial sum operators $S_{j_{1},j_{2}}$ by
\begin{align*}
S_{j_{1},j_{2}}(f)=f\ast (\Delta^{(1)}[2^{j_{1}-1}]\phi^{(1)}\otimes \Delta^{(2)}[2^{j_{2}-1}]\phi^{(2)}).
\end{align*}
Next, we define the difference operators $\Delta_{k_{1},k_{2}}$ and $\Delta_{k_{1},k_{2}}^{N(j_{1}),N(j_{2})}$ by
\begin{align*}
\Delta_{k_{1},k_{2}}f=f\ast \Psi_{k_{1},k_{2}},\
\Delta_{k_{1},k_{2}}^{N(j_{1}),N(j_{2})}f=f\ast \Psi_{k_{1},k_{2}}^{N(j_{1}),N(j_{2})},
\end{align*}
where we denote
\begin{align*}
\Psi_{k_{1},k_{2}}=\Psi_{k_{1}}^{(1)}\otimes \Psi_{k_{2}}^{(2)},\
\Psi_{k_{1},k_{2}}^{N(j_{1}),N(j_{2})}=\sum_{\ell_{1}=N(j_{1}-1)+1}^{N(j_{1})}\sum_{\ell_{2}=N(j_{2}-1)+1}^{N(j_{2})}\Psi_{k_{1}-\ell_{1},k_{2}-\ell_{2}}.
\end{align*}
We also define the mixed difference operators $(S\Delta)_{k_{1},k_{2}}^{N(j_{2})}$ and $(\Delta S)_{k_{1},k_{2}}^{N(j_{1})}$ by
$$(S\Delta)_{k_{1},k_{2}}^{N(j_{2})}f:=f\ast \big(\Delta^{(1)}[2^{k_{1}-1}]\phi^{(1)}\otimes (\Delta^{(2)}[2^{k_{2}-N(j_{2})-1}]\phi^{(2)}-\Delta^{(2)}[2^{k_{2}-N(j_{2}-1)-1}]\phi^{(2)})\big),$$
$$(\Delta S)_{k_{1},k_{2}}^{N(j_{1})}f:=f\ast \big((\Delta^{(1)}[2^{k_{1}-N(j_{1})-1}]\phi^{(1)}-\Delta^{(1)}[2^{k_{1}-N(j_{1}-1)-1}]\phi^{(1)})\otimes \Delta^{(2)}[2^{k_{2}-1}]\phi^{(2)} \big).$$
Then we have the following inequality
\begin{align*}
T_{k_{1},k_{2}}=T_{k_{1},k_{2}}S_{k_{1},k_{2}}+\sum_{j_{1}=1}^{\infty}T_{k_{1},k_{2}}(\Delta S)_{k_{1},k_{2}}^{N(j_{1})}+\sum_{j_{2}=1}^{\infty}T_{k_{1},k_{2}}(S\Delta)_{k_{1},k_{2}}^{N(j_{2})}+\sum_{j_{1}=1}^{\infty}\sum_{j_{2}=1}^{\infty}T_{k_{1},k_{2}}\Delta_{k_{1},k_{2}}^{N(j_{1}),N(j_{2})}.
\end{align*}
In this way, we get $$T=\sum_{j_{1}=0}^{\infty}\sum_{j_{2}=0}^{\infty}\tilde{T}_{j_{1},j_{2}}=\sum_{j_{1}=0}^{\infty}\sum_{j_{2}=0}^{\infty}\tilde{T}_{j_{1},j_{2}}^{N},$$ where
\begin{align*}
\tilde{T}_{0,0}:=\tilde{T}_{0,0}^{N}:=\sum_{k_{1}\in\mathbb{Z}}\sum_{k_{2}\in\mathbb{Z}}T_{k_{1},k_{2}}S_{k_{1},k_{2}},\
\end{align*}
and for $j_{1},j_{2}\geq 1$,
\begin{align*}
\tilde{T}_{0,j_{1}}:=\sum_{k_{1}\in\mathbb{Z}}\sum_{k_{2}\in\mathbb{Z}}T_{k_{1},k_{2}}(\Delta S)_{k_{1},k_{2}}^{j_{1}},\quad \tilde{T}_{0,j_{1}}^{N}:=\sum_{k_{1}\in\mathbb{Z}}\sum_{k_{2}\in\mathbb{Z}}T_{k_{1},k_{2}}(\Delta S)_{k_{1},k_{2}}^{N(j_{1})}=\sum_{\ell_{1}=N(j_{1}-1)+1}^{N(j_{1})}\tilde{T}_{0,\ell_{1}};
\end{align*}
\begin{align*}
\tilde{T}_{j_{2},0}:=\sum_{k_{1}\in\mathbb{Z}}\sum_{k_{2}\in\mathbb{Z}}T_{k_{1},k_{2}}( S\Delta)_{k_{1},k_{2}}^{j_{2}},\quad \tilde{T}_{j_{2},0}^{N}:=\sum_{k_{1}\in\mathbb{Z}}\sum_{k_{2}\in\mathbb{Z}}T_{k_{1},k_{2}}( S\Delta)_{k_{1},k_{2}}^{N(j_{2})}=\sum_{\ell_{2}=N(j_{2}-1)+1}^{N(j_{2})}\tilde{T}_{\ell_{2},0};
\end{align*}
\begin{align*}
\tilde{T}_{j_{1},j_{2}}:=\sum_{k_{1}\in\mathbb{Z}}\sum_{k_{2}\in\mathbb{Z}}T_{k_{1},k_{2}}\Delta_{k_{1}-j_{1},k_{2}-j_{2}},\quad
\tilde{T}_{j_{1},j_{2}}^{N}:=\sum_{k_{1}\in\mathbb{Z}}\sum_{k_{2}\in\mathbb{Z}}T_{k_{1},k_{2}}\Delta_{k_{1},k_{2}}^{N(j_{1}),N(j_{2})}=\sum_{\ell_{1}=N(j_{1}-1)+1}^{N(j_{1})}\sum_{\ell_{2}=N(j_{2}-1)+1}^{N(j_{2})}\tilde{T}_{\ell_{1},\ell_{2}}.
\end{align*}

Next, we show that for each $j_1,j_2$, the operator $\tilde{T}_{j_{1},j_{2}}^{N}$ is bounded on $L^2$ with the
operator norm dominated by $C_{\mathbb{Q}_1,\mathbb{Q}_2,q}2^{-\alpha N(j_{1}-1)}2^{-\alpha N(j_{2}-1)}\|K^{0}\|_{q}$.
\begin{proposition}\label{pppro}
Let $q>1$. Then there exist constants $C_{\mathbb{Q}_1,\mathbb{Q}_2,q}>0$ and $\alpha>0$ such that for any $j_{1},j_{2}\geq 0$,
\begin{align*}
\|\tilde{T}_{j_{1},j_{2}}(f)\|_{2}\leq C_{\mathbb{Q}_1,\mathbb{Q}_2,q}2^{-\alpha j_{1}}2^{-\alpha j_{2}}\|K^{0}\|_{q}\|f\|_{2}
\end{align*}
and for any $j_{1},j_{2}\geq 1$,
\begin{align*}
\|\tilde{T}_{j_{1},j_{2}}^{N}(f)\|_{2}\leq C_{\mathbb{Q}_1,\mathbb{Q}_2,q}2^{-\alpha N(j_{1}-1)}2^{-\alpha N(j_{2}-1)}\|K^{0}\|_{q}\|f\|_{2}.
\end{align*}
\end{proposition}
%
\begin{proof}
It suffices to show the first estimate when $j_{1},j_{2}\geq 1$. For other cases we can repeat the argument in the setting of one-parameter.  For simplicity, we set
\begin{align*}
T_{j_{1},j_{2},k_{1},k_{2}}(f):=T_{k_{1},k_{2}}\Delta_{k_{1}-j_{1},k_{2}-j_{2}}(f),
\end{align*}
then $T_{j_{1},j_{2}}(f)=\sum_{k_{1}\in\mathbb{Z}}\sum_{k_{2}\in\mathbb{Z}}T_{j_{1},j_{2},k_{1},k_{2}}(f)$. By Cotlar-Knapp-Stein Lemma, it suffices to show that:
\begin{align}\label{ggoall}
\|T_{j_{1},j_{2},k_{1},k_{2}}^{*}T_{j_{1},j_{2},k_{1}^{\prime},k_{2}^{\prime}}\|_{2\rightarrow 2}+\|T_{j_{1},j_{2},k_{1}^{\prime},k_{2}^{\prime}}T_{j_{1},j_{2},k_{1},k_{2}}^{*}\|_{2\rightarrow 2}\leq C2^{-2\alpha j_{1}}2^{-2\alpha j_{2}}2^{-c|k_{1}-k_{1}^{\prime}|}2^{-c|k_{2}-k_{2}^{\prime}|}\|K^{0}\|_{q}^{2},
\end{align}
for some constants $C,c>0$ and $\alpha>0$. We only estimate the first term, since the second term is similar. Note that
\begin{align}\label{blp}
T_{j_{1},j_{2},k_{1},k_{2}}^{*}T_{j_{1},j_{2},k_{1}^{\prime},k_{2}^{\prime}}(f)
&=f\ast\Psi_{k_{1}^{\prime}-j_{1},k_{2}^{\prime}-j_{2}}\ast A_{k_{1}^{\prime},k_{2}^{\prime}}K^{0}\ast A_{k_{1},k_{2}}\tilde{K^{0}}\ast \Psi_{k_{1}-j_{1},k_{2}-j_{2}}\nonumber\\
&=\sum_{\ell_{1},\ell_{2},\ell_{1}^{\prime},\ell_{2}^{\prime}\in\mathbb{Z}}f\ast\Psi_{k_{1}^{\prime}-j_{1},k_{2}^{\prime}-j_{2}}\ast A_{k_{1}^{\prime},k_{2}^{\prime}}K^{0}\ast\Psi_{\ell_{1}^{\prime},\ell_{2}^{\prime}}\ast\Psi_{\ell_{1},\ell_{2}}\ast A_{k_{1},k_{2}}\tilde{K^{0}}\ast \Psi_{k_{1}-j_{1},k_{2}-j_{2}}.
\end{align}

On the one hand, by Lemma 1 in \cite{sato2} and its duality version,
\begin{align}\label{gm1}
&\|(f\ast\Psi_{k_{1}^{\prime}-j_{1},k_{2}^{\prime}-j_{2}})\ast (A_{k_{1}^{\prime},k_{2}^{\prime}}K^{0}\ast\Psi_{\ell_{1}^{\prime},\ell_{2}^{\prime}})\ast(\Psi_{\ell_{1},\ell_{2}}\ast A_{k_{1},k_{2}}\tilde{K^{0}})\ast \Psi_{k_{1}-j_{1},k_{2}-j_{2}}\|_{2}\nonumber\\
&\leq C\|(f\ast\Psi_{k_{1}^{\prime}-j_{1},k_{2}^{\prime}-j_{2}})\ast (A_{k_{1}^{\prime},k_{2}^{\prime}}K^{0}\ast\Psi_{\ell_{1}^{\prime},\ell_{2}^{\prime}})\ast(\Psi_{\ell_{1},\ell_{2}}\ast A_{k_{1},k_{2}}\tilde{K^{0}})\|_{2}\nonumber\\
&\leq C2^{-\beta|k_{1}-\ell_{1}|}2^{-\beta|k_{2}-\ell_{2}|}\|K^{0}\|_{q}\|(f\ast\Psi_{k_{1}^{\prime}-j_{1},k_{2}^{\prime}-j_{2}})\ast (A_{k_{1}^{\prime},k_{2}^{\prime}}K^{0}\ast\Psi_{\ell_{1}^{\prime},\ell_{2}^{\prime}})\|_{2}\nonumber\\
&\leq C2^{-\beta|k_{1}-\ell_{1}|}2^{-\beta|k_{2}-\ell_{2}|}2^{-\beta|k_{1}^{\prime}-\ell_{1}^{\prime}|}2^{-\beta|k_{2}^{\prime}-\ell_{2}^{\prime}|}\|K^{0}\|_{q}^{2}\|f\ast\Psi_{k_{1}^{\prime}-j_{1},k_{2}^{\prime}-j_{2}}\|_{2}\nonumber\\
&\leq C2^{-\beta|k_{1}-\ell_{1}|}2^{-\beta|k_{2}-\ell_{2}|}2^{-\beta|k_{1}^{\prime}-\ell_{1}^{\prime}|}2^{-\beta|k_{2}^{\prime}-\ell_{2}^{\prime}|}\|K^{0}\|_{q}^{2}\|f\|_{2}
\end{align}
for some $C>0$ and $\beta>0$.

On the other hand, by the smoothness and the cancellation properties of $\Psi_{\ell_{1}^{\prime},\ell_{2}^{\prime}}$ and $\Psi_{\ell_{1},\ell_{2}}$, we see that $\|\Psi_{\ell_{1}^{\prime},\ell_{2}^{\prime}}\ast\Psi_{\ell_{1},\ell_{2}}\|_{1}\leq C2^{-|\ell_{1}-\ell_{1}^{\prime}|}2^{-|\ell_{2}-\ell_{2}^{\prime}|}$. Therefore,
\begin{align}\label{gm2}
&\|(f\ast\Psi_{k_{1}^{\prime}-j_{1},k_{2}^{\prime}-j_{2}}\ast A_{k_{1}^{\prime},k_{2}^{\prime}}K^{0})\ast(\Psi_{\ell_{1}^{\prime},\ell_{2}^{\prime}}\ast\Psi_{\ell_{1},\ell_{2}})\ast (A_{k_{1},k_{2}}\tilde{K^{0}}\ast \Psi_{k_{1}-j_{1},k_{2}-j_{2}})\|_{2}\nonumber\\
&\leq C2^{-\beta|j_{1}|}2^{-\beta|j_{2}|}\|K^{0}\|_{q}\|(f\ast\Psi_{k_{1}^{\prime}-j_{1},k_{2}^{\prime}-j_{2}}\ast A_{k_{1}^{\prime},k_{2}^{\prime}}K^{0})\ast(\Psi_{\ell_{1}^{\prime},\ell_{2}^{\prime}}\ast\Psi_{\ell_{1},\ell_{2}})\|_{2}\nonumber\\
&\leq C2^{-\beta|j_{1}|}2^{-\beta|j_{2}|}2^{-|\ell_{1}^{\prime}-\ell_{1}|}2^{-|\ell_{2}^{\prime}-\ell_{2}|}\|K^{0}\|_{q}\|f\ast\Psi_{k_{1}^{\prime}-j_{1},k_{2}^{\prime}-j_{2}}\ast A_{k_{1}^{\prime},k_{2}^{\prime}}K^{0}\|_{2}\nonumber\\
&\leq C2^{-2\beta|j_{1}|}2^{-2\beta|j_{2}|}2^{-|\ell_{1}^{\prime}-\ell_{1}|}2^{-|\ell_{2}^{\prime}-\ell_{2}|}\|K^{0}\|_{q}\|f\|_{2}.
\end{align}

Taking geometric mean of \eqref{gm1} and \eqref{gm2} and then applying triangle inequality, we see that
\begin{align*}
&\|f\ast\Psi_{k_{1}^{\prime}-j_{1},k_{2}^{\prime}-j_{2}}\ast A_{k_{1}^{\prime},k_{2}^{\prime}}K^{0}\ast\Psi_{\ell_{1}^{\prime},\ell_{2}^{\prime}}\ast\Psi_{\ell_{1},\ell_{2}}\ast A_{k_{1},k_{2}}\tilde{K^{0}}\ast \Psi_{k_{1}-j_{1},k_{2}-j_{2}}\|_{2}\\
&\leq C2^{-cj_{1}}2^{-cj_{2}}2^{-c|k_{1}-k_{1}^{\prime}|}2^{-c|k_{1}-\ell_{1}|}2^{-c|k_{2}-\ell_{2}|}2^{-c|k_{2}-k_{2}^{\prime}|}2^{-c|k_{1}^{\prime}-\ell_{1}^{\prime}|}2^{-c|k_{2}^{\prime}-\ell_{2}^{\prime}|}\|K^{0}\|_{q}^{2}\|f\|_{2}.
\end{align*}
This, in combination with the equality \eqref{blp}, implies the estimate \eqref{ggoall}. This ends the proof of Proposition \ref{pppro}.
\end{proof}

We now further prove that the kernel $K_{j_{1},j_{2}}^{N}(x,y)$ of the operator $\tilde{T}_{j_{1},j_{2}}^{N}$
satisfies the Dini-type Calder\'on--Zygmund kernel condition with the Dini$_1$ condition \cite[Section 5]{AMV}.
\begin{lemma}\label{lem6.3}
For any $q>1$, there exists a constant $C_{\mathbb{Q}_1,\mathbb{Q}_2,q}>0$ such that the kernel $K_{j_{1},j_{2}}^{N}(x,y)$  satisfies the size estimate
\begin{align*}
|K_{j_{1},j_{2}}^{N}(x,y)|\leq C_{\mathbb{Q}_1,\mathbb{Q}_2,q}2^{\frac{N(j_{1})\mathbb{Q}_1}{q}}2^{\frac{N(j_{2})\mathbb{Q}_2}{q}}\frac{1}{d_{1}(x_{1},y_{1})^{\mathbb{Q}_{1}}}\frac{1}{d_{2}(x_{2},y_{2})^{\mathbb{Q}_{2}}}\|K^{0}\|_{q},
\end{align*}
the H\"{o}lder estimate
\begin{align*}
&|K_{j_{1},j_{2}}^{N}(x,y)-K_{j_{1},j_{2}}^{N}((x_{1},x_{2}^{\prime}),y)-K_{j_{1},j_{2}}^{N}((x_{1}^{\prime},x_{2}),y)+K_{j_{1},j_{2}}^{N}(x^{\prime},y)|\\
&\leq C_{\mathbb{Q}_1,\mathbb{Q}_2,q}\omega_{j_{1},j_{2}}^{N,1}\left(\frac{d_{1}(x_{1},x_{1}^{\prime})}{d_{1}(x_{1},y_{1})}\right)\frac{1}{d_{1}(x_{1},y_{1})^{\mathbb{Q}_{1}}}\omega_{j_{1},j_{2}}^{N,2}\left(\frac{d_{2}(x_{2},x_{2}^{\prime})}{d_{2}(x_{2},y_{2})}\right)\frac{1}{d_{2}(x_{2},y_{2})^{\mathbb{Q}_{2}}},
\end{align*}
whenever $d_{1}(x_{1},y_{1})\geq 2A_{0}^{(1)}d_{1}(x_{1},x_{1}^{\prime})$ and $d_{2}(x_{2},y_{2})\geq 2A_{0}^{(2)}d_{2}(x_{2},x_{2}^{\prime})$, and the mixed H\"{o}lder and size estimates
\begin{align*}
|K_{j_{1},j_{2}}^{N}(x,y)-K_{j_{1},j_{2}}^{N}((x_{1}^{\prime},x_{2}),y)|\leq C_{\mathbb{Q}_1,\mathbb{Q}_2,q}\omega_{j_{1},j_{2}}^{N,1}\left(\frac{d_{1}(x_{1},x_{1}^{\prime})}{d_{1}(x_{1},y_{1})}\right)\frac{1}{d_{1}(x_{1},y_{1})^{\mathbb{Q}_{1}}}\frac{1}{d_{2}(x_{2},y_{2})^{\mathbb{Q}_{2}}},
\end{align*}
whenever $d_{1}(x_{1},y_{1})\geq 2A_{0}^{(1)}d_{1}(x_{1},x_{1}^{\prime})$ and
\begin{align*}
|K_{j_{1},j_{2}}^{N}(x,y)-K_{j_{1},j_{2}}^{N}((x_{1},x_{2}^{\prime}),y)|\leq C_{\mathbb{Q}_1,\mathbb{Q}_2,q}\frac{1}{d_{1}(x_{1},y_{1})^{\mathbb{Q}_{1}}}\omega_{j_{1},j_{2}}^{N,2}\left(\frac{d(x_{2},x_{2}^{\prime})}{d(x_{2},y_{2})}\right)\frac{1}{d_{2}(x_{2},y_{2})^{\mathbb{Q}_{2}}},
\end{align*}
whenever $d_{2}(x_{2},y_{2})\geq 2A_{0}^{(2)}d_{2}(x_{2},x_{2}^{\prime})$.

Moreover, for each $k=1,2$, $\omega_{j_{1},j_{2}}^{N,k}$ satisfies the modified {\rm Dini}$_1$ condition:
\begin{align}\label{dini1}
\|\omega_{j_{1},j_{2}}^{N,k}\|_{{\rm Dini}_1}\leq C2^{\frac{N(j_{1})\mathbb{Q}_{1}}{q}}2^{\frac{N(j_{2})\mathbb Q_{2}}{q}}\|K^{0}\|_{q}(1+N(j_{k}))^{2}.
\end{align}
\end{lemma}
\begin{proof}
In order to obtain these estimates for $K_{j_{1},j_{2}}^{N}$, we first study the kernel of $T_{k_{1},k_{2}}S_{k_{1}-N(j_{1}),k_{2}-N(j_{2})}$. A direct calculation implies
\begin{align*}
&|A_{k_{1},k_{2}}K^{0}(x_{1},x_{2})|\\
&=2^{-k_{1}}2^{-k_{2}}\left|\int_{0}^{\infty}\int_{0}^{\infty}\varphi(2^{-k_{1}}t_{1})\varphi(2^{-k_{2}}t_{2})\Delta[t_{1},t_{2}]K^{0}(x_{1},x_{2})dt_{1}dt_{2}\right|\\
&\leq C\rho_{1}(x_{1})^{-\mathbb{Q}_{1}}\rho_{2}(x_{2})^{-\mathbb{Q}_{2}}|K(\rho_{1}(x_{1})^{-1}\circ_{1} x_{1},\rho_{2}(x_{2})^{-1}\circ_{2}x_{2})|\chi_{2^{k_{1}-1}\leq\rho_{1}(x_{1})\leq 2^{k_{1}+2}}(x_{1})\chi_{2^{k_{2}-1}\leq\rho_{2}(x_{1})\leq 2^{k_{2}+2}}(x_{2}).
\end{align*}
This, in combination with the observation that for $i=1,2$, $\supp\phi^{(i)}\subset \{x_{i}\in\HH_{i}:\rho_{i}(x_{i})\leq \frac{1}{100}\}$, indicates
\begin{align*}
&|(\Delta^{(1)}[2^{k_{1}-N(j_{1})-1}]\phi^{(1)}\otimes \Delta^{(2)}[2^{k_{2}-N(j_{2})-1}]\phi^{(2)})\ast A_{k_{1},k_{2}}K^{0}(x_{1},x_{2})|\\
&\leq C\int_{\HH_{1}}\int_{\HH_{2}}2^{-[k_{1}-N(j_{1})-1]\mathbb{Q}_{1}}|\phi^{(1)}(2^{-[k_{1}-N(j_{1})-1]}\circ_{1}x_{1}u^{-1})|\rho_{1}(u)^{-\mathbb{Q}_{1}}\chi_{2^{k_{1}-1}\leq\rho_{1}(u)\leq 2^{k_{1}+2}}(u)\\
&\quad\times2^{-[k_{2}-N(j_{2})-1]\mathbb{Q}_{2}}|\phi^{(2)}(2^{-[k_{2}-N(j_{2})-1]}\circ_{2}x_{2}v^{-1})|\rho_{2}(v)^{-\mathbb{Q}_{2}}\chi_{2^{k_{2}-1}\leq\rho_{2}(v)\leq 2^{k_{2}+2}}(v)\\
&\hspace{9.0cm}\times |K(\rho_{1}(u)^{-1}\circ_{1} u,\rho_{2}(v)^{-1}\circ_{2}v)|dudv\\
&\leq C2^{-[k_{1}-N(j_{1})-1]\mathbb{Q}_{1}}2^{-[k_{2}-N(j_{2})-1]\mathbb{Q}_{2}}\rho_{1}(x_{1})^{-\mathbb{Q}_{1}}\rho_{2}(x_{2})^{-\mathbb{Q}_{2}}\chi_{2^{k_{1}-2}\leq\rho_{1}(x_{1})\leq 2^{k_{1}+3}}(x_{1})\chi_{2^{k_{2}-2}\leq\rho_{2}(x_{2})\leq 2^{k_{2}+3}}(x_{2})\\
&\quad\times\int_{2^{k_{1}-1}\leq\rho_{1}(u)\leq 2^{k_{1}+2}}\int_{2^{k_{2}-1}\leq\rho_{2}(v)\leq 2^{k_{2}+2}}|\phi^{(1)}(2^{-[k_{1}-N(j_{1})-1]}\circ_{1}x_{1}u^{-1})||\phi^{(2)}(2^{-[k_{2}-N(j_{2})-1]}\circ_{2}x_{2}v^{-1})|\\
&\hspace{9.0cm}\times|K(\rho_{1}(u)^{-1}\circ_{1} u,\rho_{2}(v)^{-1}\circ_{2}v)|dudv.
\end{align*}
By H\"{o}lder's inequality, for $1/q+1/q^{\prime}=1$, we have
\begin{align*}
&\int_{2^{k_{1}-1}\leq\rho_{1}(u)\leq 2^{k_{1}+2}}\int_{2^{k_{2}-1}\leq\rho_{2}(v)\leq 2^{k_{2}+2}}|\phi^{(1)}(2^{-[k_{1}-N(j_{1})-1]}\circ_{1}x_{1}u^{-1})||\phi^{(2)}(2^{-[k_{2}-N(j_{2})-1]}\circ_{2}x_{2}v^{-1})|\\
&\hspace{9.0cm}\times|K(\rho_{1}(u)^{-1}\circ_{1} u,\rho_{2}(v)^{-1}\circ_{2}v)|dudv\\
&\leq \left(\int_{2^{k_{1}-1}\leq\rho_{1}(u)\leq 2^{k_{1}+2}}\int_{2^{k_{2}-1}\leq\rho_{2}(v)\leq 2^{k_{2}+2}}|K(\rho_{1}(u)^{-1}\circ_{1} u,\rho_{2}(v)^{-1}\circ_{2}v)|^{q}dudv\right)^{1/q}\\
&\hspace{2.0cm}\times\left(\int_{\HH_{1}\times\HH_{2}}|\phi^{(1)}(2^{-[k_{1}-N(j_{1})-1]}\circ_{1}x_{1}u^{-1})|^{q^{\prime}}|\phi^{(2)}(2^{-[k_{2}-N(j_{2})-1]}\circ_{2}x_{2}v^{-1})|^{q^{\prime}}dudv\right)^{1/q^{\prime}}\\
&\leq C2^{\frac{k_{1}\mathbb{Q}_{1}}{q}}2^{\frac{k_{2}\mathbb{Q}_{2}}{q}}2^{\frac{[k_{1}-N(j_{1})-1]\mathbb{Q}_{1}}{q^{\prime}}}2^{\frac{[k_{2}-N(j_{2})-1]\mathbb{Q}_{2}}{q^{\prime}}}\|K^{0}\|_{q}.
\end{align*}
Combining the two estimates we obtain above, we see that
\begin{align}\label{sizecondition}
&\sum_{k_{1}\in\mathbb{Z}}\sum_{k_{2}\in\mathbb{Z}}|(\Delta^{(1)}[2^{k_{1}-N(j_{1})-1}]\phi^{(1)}\otimes \Delta^{(2)}[2^{k_{2}-N(j_{2}-1}]\phi^{(2)})\ast A_{k_{1},k_{2}}K^{0}(x_{1},x_{2})|\nonumber\\
&\leq C2^{\frac{N(j_{1})\mathbb{Q}_{1}}{q}}2^{\frac{N(j_{2})\mathbb{Q}_{2}}{q}}\rho_{1}(x_{1})^{-\mathbb{Q}_{1}}\rho_{2}(x_{2})^{-\mathbb{Q}_{2}}\|K^{0}\|_{q}.
\end{align}
Similarly, we obtain the gradient estimates as follows.
\begin{align}\label{smoothcondition1}
&\sum_{k_{1}\in\mathbb{Z}}\sum_{k_{2}\in\mathbb{Z}}|\nabla_{x_{1}}(\Delta^{(1)}[2^{k_{1}-N(j_{1})-1}]\phi^{(1)}\otimes \Delta^{(2)}[2^{k_{2}-N(j_{2})-1}]\phi^{(2)})\ast A_{k_{1},k_{2}}K^{0}(x_{1},x_{2})|\nonumber\\
&\leq C2^{N(j_{1})\Big(1+\frac{\mathbb{Q}_{1}}{q}\Big)}2^{\frac{N(j_{2})\mathbb{Q}_{2}}{q}}\rho_{1}(x_{1})^{-\mathbb{Q}_{1}-1}\rho_{2}(x_{2})^{-\mathbb{Q}_{2}}\|K^{0}\|_{q},
\end{align}
and
\begin{align}\label{smoothcondition2}
&\sum_{k_{1}\in\mathbb{Z}}\sum_{k_{2}\in\mathbb{Z}}|\nabla_{x_{2}}(\Delta^{(1)}[2^{k_{1}-N(j_{1})-1}]\phi^{(1)}\otimes \Delta^{(2)}[2^{k_{2}-N(j_{2})-1}]\phi^{(2)})\ast A_{k_{1},k_{2}}K^{0}(x_{1},x_{2})|\nonumber\\
&\leq C2^{\frac{N(j_{1})\mathbb{Q}_{1}}{q}}2^{N(j_{2})\Big(1+\frac{\mathbb{Q}_{2}}{q}\Big)}\rho_{1}(x_{1})^{-\mathbb{Q}_{1}}\rho_{2}(x_{2})^{-\mathbb{Q}_{2}-1}\|K^{0}\|_{q}.
\end{align}
We also have mixed gradient estimate
\begin{align*}
&\sum_{k_{1}\in\mathbb{Z}}\sum_{k_{2}\in\mathbb{Z}}|\nabla_{x_{1}}\nabla_{x_{2}}(\Delta^{(1)}[2^{k_{1}-N(j_{1})-1}]\phi^{(1)}\otimes \Delta^{(2)}[2^{k_{2}-N(j_{2})-1}]\phi^{(2)})\ast A_{k_{1},k_{2}}K^{0}(x_{1},x_{2})|\\
&\leq C2^{N(j_{1})\Big(1+\frac{\mathbb{Q}_{1}}{q}\Big)}2^{N(j_{2})\Big(1+\frac{\mathbb{Q}_{2}}{q}\Big)}\rho_{1}(x_{1})^{-\mathbb{Q}_{1}-1}\rho_{2}(x_{2})^{-\mathbb{Q}_{2}-1}\|K^{0}\|_{q}.
\end{align*}
From the triangle inequality and $N(j-1)<N(j)$ we can see that the kernel
\begin{align*}
K_{j_{1},j_{2}}^{N}&:=\sum_{k_{1}\in\mathbb{Z}}\sum_{k_{2}\in\mathbb{Z}}\big((\Delta[2^{k-N(j_{1})-1}]\phi^{(1)}-\Delta[2^{k-N(j_{1}-1)-1}\phi^{(1)}])\\
&\quad \otimes (\Delta[2^{k-N(j_{2})-1}]\phi^{(2)}-\Delta[2^{k-N(j_{2}-1)-1}\phi^{(2)}])\big)\ast A_{k_{1},k_{2}}K^{0}
\end{align*}
satisfies the same estimates \eqref{sizecondition}, \eqref{smoothcondition1} and \eqref{smoothcondition2}.
That is,
\begin{align*}
|K_{j_{1},j_{2}}^{N}(x_{1},x_{2},y_{1},y_{2})|&=|K_{j_{1},j_{2}}^{N}(y_{1}^{-1}x_{1},y_{2}^{-1}x_{2})|\\
&\leq C2^{\frac{N(j_{1})\mathbb{Q}_{1}}{q}}2^{\frac{N(j_{2})\mathbb{Q}_{2}}{q}}d_{1}(x_{1},y_{1})^{-\mathbb{Q}_{1}}d_{2}(x_{2},y_{2})^{-\mathbb{Q}_{2}}\|K^{0}\|_{q}.
\end{align*}
\begin{align*}
|\nabla_{(x_{1},y_{1})}K_{j_{1},j_{2}}^{N}(x_{1},x_{2},y_{1},y_{2})|\leq C2^{N(j_{1})\Big(1+\frac{\mathbb{Q}_{1}}{q}\Big)}2^{\frac{N(j_{2})\mathbb{Q}_{2}}{q}}d_{1}(x_{1},y_{1})^{-\mathbb{Q}_{1}-1}d_{2}(x_{2},y_{2})^{-\mathbb{Q}_{2}}\|K^{0}\|_{q}.
\end{align*}
\begin{align*}
|\nabla_{(x_{2},y_{2})}K_{j_{1},j_{2}}^{N}(x_{1},x_{2},y_{1},y_{2})|\leq C2^{\frac{N(j_{1})\mathbb{Q}_{1}}{q}}2^{N(j_{2})\Big(1+\frac{\mathbb{Q}_{2}}{q}\Big)}d_{1}(x_{1},y_{1})^{-\mathbb{Q}_{1}}d_{2}(x_{2},y_{2})^{-\mathbb{Q}_{2}-1}\|K^{0}\|_{q}.
\end{align*}
\begin{align*}
|\nabla_{(x_{1},y_{1})}\nabla_{(x_{2},y_{2})}K_{j_{1},j_{2}}^{N}(x_{1},x_{2},y_{1},y_{2})|\leq C2^{N(j_{1})\Big(1+\frac{\mathbb{Q}_{1}}{q}\Big)}2^{N(j_{2})\Big(1+\frac{\mathbb{Q}_{2}}{q}\Big)}d_{1}(x_{1},y_{1})^{-\mathbb{Q}_{1}-1}d_{2}(x_{2},y_{2})^{-\mathbb{Q}_{2}-1}\|K^{0}\|_{q}.
\end{align*}
(For $j_{1}=0$ or $j_{2}=0$, the subtraction is not even needed.)
Besides, for $d_{1}(x_{1},y_{1})\geq 2A_{0}^{(1)}d(x_{1},x_{1}^{\prime})$, by mean value theorem on homogeneous groups,
\begin{align*}
|K_{j_{1},j_{2}}^{N}(x,y)-K_{j_{1},j_{2}}^{N}((x_{1}^{\prime},x_{2}),y)|&=|K_{j_{1},j_{2}}^{N}(y_{1}^{-1}x_{1},y_{2}^{-1}x_{2})-K_{j_{1},j_{2}}^{N}(y_{1}^{-1}x_{1}^{\prime},y_{2}^{-1}x_{2})|\\
&\leq C2^{N(j_{1})\Big(1+\frac{\mathbb{Q}_{1}}{q}\Big)}2^{\frac{N(j_{2})\mathbb{Q}_{2}}{q}}d_{1}(x_{1},y_{1})^{-\mathbb{Q}_{1}-1}d_{1}(x,x^{\prime})d_{2}(x_{2},y_{2})^{-\mathbb{Q}_{2}}\|K^{0}\|_{q}.
\end{align*}
By triangle inequality, we also have
\begin{align*}
|K_{j_{1},j_{2}}^{N}(x,y)-K_{j_{1},j_{2}}^{N}((x_{1}^{\prime},x_{2}),y)|\leq C2^{\frac{N(j_{1})\mathbb{Q}_{1}}{q}}2^{\frac{N(j_{2})\mathbb{Q}_{2}}{q}}d_{1}(x_{1},y_{1})^{-\mathbb{Q}_{1}}d_{2}(x_{2},y_{2})^{-\mathbb{Q}_{2}}\|K^{0}\|_{q}.
\end{align*}
Combining the two estimates we obtain above and by symmetry, we conclude that
\begin{align*}
|K_{j_{1},j_{2}}^{N}(x,y)-K_{j_{1},j_{2}}^{N}((x_{1}^{\prime},x_{2}),y)|\leq C\left(\frac{d_{1}(x_{1},x_{1}^{\prime})}{d_{1}(x_{1},y_{1})}\right)\frac{1}{d_{1}(x_{1},y_{1})^{\mathbb{Q}_{1}}}\frac{1}{d_{2}(x_{2},y_{2})^{\mathbb{Q}_{2}}},
\end{align*}
where
\begin{align}\label{pw1}
\omega_{j_{1},j_{2}}^{N,1}(t)\leq C2^{\frac{N(j_{1})\mathbb{Q}_{1}}{q}}2^{\frac{N(j_{2})\mathbb{Q}_{2}}{q}}\|K^{0}\|_{q}\min\{1,2^{N(j_{1})}t\}.
\end{align}
Similarly, whenever $d(x_{2},y_{2})\geq 2A_{0}^{(2)}d(x_{2},x_{2}^{\prime})$, we have
\begin{align*}
|K_{j_{1},j_{2}}^{N}(x,y)-K_{j_{1},j_{2}}^{N}((x_{1},x_{2}^{\prime}),y)|\leq C\frac{1}{d_{1}(x_{1},y_{1})^{\mathbb{Q}_{1}}}\omega_{j_{1},j_{2}}^{N,2}\left(\frac{d(x_{2},x_{2}^{\prime})}{d(x_{2},y_{2})}\right)\frac{1}{d_{2}(x_{2},y_{2})^{\mathbb{Q}_{2}}},
\end{align*}
where
\begin{align}\label{pw2}
\omega_{j_{1},j_{2}}^{N,2}(t)\leq C2^{\frac{N(j_{1})\mathbb{Q}_{1}}{q}}2^{\frac{N(j_{2})\mathbb{Q}_{2}}{q}}\|K^{0}\|_{q}\min\{1,2^{N(j_{2})}t\}.
\end{align}
Furthermore,
\begin{align*}
&|K_{j_{1},j_{2}}^{N}(x,y)-K_{j_{1},j_{2}}^{N}((x_{1},x_{2}^{\prime}),y)-K_{j_{1},j_{2}}^{N}((x_{1}^{\prime},x_{2}),y)+K_{j_{1},j_{2}}^{N}(x^{\prime},y)|\\
&\leq C\omega_{j_{1},j_{2}}^{N,1}\left(\frac{d_{1}(x_{1},x_{1}^{\prime})}{d_{1}(x_{1},y_{1})}\right)\frac{1}{d_{1}(x_{1},y_{1})^{\mathbb{Q}_{1}}}\omega_{j_{1},j_{2}}^{N,2}\left(\frac{d_{2}(x_{2},x_{2}^{\prime})}{d_{2}(x_{2},y_{2})}\right)\frac{1}{d_{2}(x_{2},y_{2})^{\mathbb{Q}_{2}}},
\end{align*}
whenever $d_{1}(x_{1},y_{1})\geq 2A_{0}^{(1)}d_{1}(x_{1},x_{1}^{\prime})$ and $d_{2}(x_{2},y_{2})\geq 2A_{0}^{(2)}d_{2}(x_{2},x_{2}^{\prime})$.

Now we verified that for each $k=1,2$, the estimate \eqref{dini1} holds. Indeed, it follows from the pointwise estimates \eqref{pw1} and \eqref{pw2} that
\begin{align}\label{no3}
&\int_{0}^{1}\omega_{j_{1},j_{2}}^{N,k}(t)\left(1+\log\frac{1}{t}\right)\frac{dt}{t}\nonumber\\
&\leq C2^{\frac{N(j_{1})\mathbb{Q}_{1}}{q}}2^{\frac{N(j_{2})\mathbb Q_{2}}{q}}\|K^{0}\|_{q}\left(\int_{0}^{1}\min\{1,2^{N(j_{k})}t\}\frac{dt}{t}+\int_{0}^{1}\min\{1,2^{N(j_{k})}t\}\log\frac{1}{t}\frac{dt}{t}\right).
\end{align}
Note that
\begin{align}\label{no1}
\int_{0}^{1}\min\{1,2^{N(j_{k})}t\}\frac{dt}{t}\leq C \left(\int_{0}^{2^{-N(j_{k})}}2^{N(j_{k})}t+\int_{2^{-N(j_{k})}}^{1}\frac{dt}{t}\right)\leq C(1+N(j_{k})).
\end{align}
Besides, integration by parts yields
\begin{align}\label{no2}
\int_{0}^{1}\min\{1,2^{N(j_{k})}t\}\log\frac{1}{t}\frac{dt}{t}&=\left(-\int_{0}^{2^{-N(j_{k})}}2^{N(j_{k})}\log t dt-\int_{2^{-N(j_{k})}}^{1}\log t\frac{dt}{t}\right)\nonumber\\
&\leq C(N(j_{k})+1)+C(N(j_{k}))^{2}\nonumber\\
&\leq C(1+N(j_{k}))^{2}.
\end{align}
Combining the estimates \eqref{no3}, \eqref{no1} and \eqref{no2}, we verify the Dini$_1$ condition \eqref{dini1} and then the proof of Lemma \ref{lem6.3} is complete.
\end{proof}

Combining the  Proposition \ref{pppro} and Lemma \ref{lem6.3}, we see that
by applying Theorem 5.35 of \cite{AMV} to our $\tilde T^N_{j_1,j_2}$ (with Dini$_1$ condition \eqref{dini1} for $\omega_{j_{1},j_{2}}^{N,1}$ and $\omega_{j_{1},j_{2}}^{N,2}$), we get the representation theorem
\begin{align}
\langle \tilde T^N_{j_1,j_2}f,g\rangle = C \mathbb E_\sigma \sum_{k=(k_1,k_2)\in \mathbb N^2} \omega_{j_{1},j_{2}}^{N,1}(2^{-k_1})\omega_{j_{1},j_{2}}^{N,2}(2^{-k_2})
\langle V_{k,\sigma}f,g \rangle,
\end{align}
where $V_{k,\sigma}$ is the standard bi-parameter dyadic Haar shifts since $\tilde T^N_{j_1,j_2}$ is paraproduct-free \cite[Lemma 5.12]{AMV}. Here the only concern is that we are working on $\HH_1\times \HH_2$ while the setting in \cite{AMV} is $\mathbb R^n \times \mathbb R^n$. In fact, one can obtain this result parallel to the Euclidean setting by using the Haar basis on space of homogeneous type constructed
in \cite{KLPW} and  the probability space and expectation in \cite{NRV}.
To be more specific, $V_{k,\sigma}$
is given by
\begin{align*}
\langle V_{k,\sigma}f,g \rangle=\sum_{i_1=0}^{k_1}\sum_{i_2=0}^{k_2} \langle S_{k_1,i_1,\sigma}^{k_2,i_2} f,g \rangle
\end{align*}
with
\begin{align*}
\langle S_{k_1,i_1,\sigma}^{k_2,i_2} f,g \rangle = \sum_{K_1,K_2}
\sum_{\substack{ I_1^{k_1} = J_1^{(i_1)}=K_1, \\ I_2^{k_2} = J_2^{(i_2)}=K_2 }}
a_{I_1J_1K_1I_2J_2K_2} \langle f, h_{I_1}h_{I_2} \rangle\langle g, h_{J_1}h_{J_2} \rangle,
\end{align*}
where $h_{I_i}$, $h_{J_i}$ are the Haar basis in $\HH_i$ (for the explicit definition, we refer to \cite{KLPW}).

From \cite[Theorem 2]{BP}, we get that
\begin{align*}
|\langle S_{k_1,i_1,\sigma}^{k_2,i_2} f,g \rangle| \lesssim 2^{-k_1-k_2-i_1-i_2}
\sum_{R\in \Lambda_{k_1,k_2,i_1,i_2}} |R| (\mathcal S_\sigma^{k_1,k_2;i_1,i_2}f)_R (\mathcal S_\sigma^{i_1,i_2;k_1,k_2}g)_R,
\end{align*}
where $\Lambda_{k_1,k_2,i_1,i_2}$ is a sparse collection of dyadic rectangles depending on $f,g$, and
$\mathcal S_\sigma^{k_1,k_2;i_1,i_2}f$ is the shifted square function \cite[equation (12)]{BP}.

Hence, we have
\begin{align}
|\langle \tilde T^N_{j_1,j_2}f,g\rangle| &\lesssim \mathbb E_\sigma \sum_{k=(k_1,k_2)\in \mathbb N^2} \omega_{j_{1},j_{2}}^{N,1}(2^{-k_1})\omega_{j_{1},j_{2}}^{N,2}(2^{-k_2}) \\
&\quad\sum_{i_1=0}^{k_1}\sum_{i_2=0}^{k_2}
2^{-k_1-k_2-i_1-i_2}  \sum_{R\in \Lambda_{k_1,k_2,i_1,i_2}} |R| (\mathcal S_\sigma^{k_1,k_2;i_1,i_2}f)_R (\mathcal S_\sigma^{i_1,i_2;k_1,k_2}g)_R.\nonumber
\end{align}
By noting that $\|\mathcal S_\sigma^{k_1,k_2;i_1,i_2}f \|_{L^2(w)}\lesssim [w]_2^4 [w]_\infty \|f\|_{L^2(w)}$ \cite[Section 5]{BP}
and that $\|M_sf\|_{L^2(w)}\lesssim [w]_2^2\|f\|_{L^2(w)}$, we get that
\begin{align*}
\sum_{R\in \Lambda_{k_1,k_2,i_1,i_2}} |R| (\mathcal S_\sigma^{k_1,k_2;i_1,i_2}f)_R (\mathcal S_\sigma^{i_1,i_2;k_1,k_2}g)_R
&\lesssim \| M_s\mathcal S_\sigma^{k_1,k_2;i_1,i_2}f \|_{L^2(w)} \| M_s\mathcal S_\sigma^{i_1,i_2;k_1,k_2}g \|_{L^2(w^{-1})}\\
&\lesssim  [w]_{2}^4 \| \mathcal S_\sigma^{k_1,k_2;i_1,i_2}f \|_{L^2(w)} \| \mathcal S_\sigma^{i_1,i_2;k_1,k_2}g \|_{L^2(w^{-1})}\\
&\lesssim  [w]_{2}^{12}[w]_\infty^2 \| f \|_{L^2(w)} \| g \|_{L^2(w^{-1})}.
\end{align*}
As a consequence, we get that
\begin{align}\label{proweig}
\| \tilde T^N_{j_1,j_2}f\|_{L^2(w)} \leq C_{\mathbb{Q}_{1},\mathbb{Q}_{2}} [w]_{A_2}^{12}[w]_{A_\infty}^2\ \|K^{0}\|_{q}^2\prod_{k=1}^22^{\frac{2N(j_{k})\mathbb{Q}_{k}}{q}}(1+N(j_{k}))^{2} \ \| f\|_{L^2(w)} .
\end{align}
Finally, similar to the proof in the one-parameter setting, by applying the interpolation theorem with change of measure to the estimates \eqref{pppro} and \eqref{proweig}, we see that there exists a constant $c_{\mathbb{Q}_{1},\mathbb{Q}_{2}}>0$ such that if $K^{0}\in L^{q}(D_{0})$ for some $q>c_{\mathbb{Q}_{1},\mathbb{Q}_{2}}(w)_{A_{2}}$ then
\begin{align*}
\|T\|_{L^{2}(w)\rightarrow L^{2}(w)}\leq C_{\mathbb{Q}_1,\mathbb{Q}_2,q}\max\{\|K^{0}\|_{q},\|K^{0}\|_{q}^{2}\}[w]_{A_2}^{12}[w]_{A_\infty}^2,
\end{align*}
for some constant $C_{\mathbb{Q}_1,\mathbb{Q}_2,q}$ independent of $w$.

The proof of  Theorem \ref{main3} is complete.
\end{proof}

\bigskip

 \noindent
 {\bf Acknowledgements:}
 J. Li would like to thank Henri Martikainen for introducing his latest result  \cite{AMV} on the representation theorem for bi-parameter Dini-type Calder\'on--Zygmund operators. Z.J. Fan would like to thank Prof. Lixin Yan for helpful discussions. The authors would like to thank Emiel. Lorist for introducing his result about sparse domination on general homogeneous spaces.

Z.J. Fan is supported by International Program for Ph.D. Candidates from Sun Yat-Sen University.
J. Li is supported by the Australian Research Council (ARC) through the
research grant DP170101060 and by Macquarie University Research Seeding Grant.


\begin{thebibliography}{99999}

\bibitem{AP} H. Al-Qassem and Y. Pan, $L^p$ boundedness for singular integrals with rough kernels on product domains, {\it Hokkaido Math. J}.  {\bf 31} (2002), 555--613.

\bibitem{AAP} A. Al-Salman, H. Al-Qassem and Y. Pan, Singular integrals on product domains, {\it Indiana Univ. Math. J}. {\bf55} (2006), 369--387.


\bibitem{AMV} E. Airta, H. Martikainen and E. Vuorinen, Modern singular integral theory with mild kernel regularity, Available at arXiv: 2006.05807 (2020).

\bibitem{BP} A. Barron and J. Pipher, Sparse domination for bi-parameter operators using square functions, Available at arXiv: 1709.05009 (2017).

\bibitem{CZ0} A.P. Calder\'{o}n and A. Zygmund,  On the existence of certain singular integrals,
       {\it Acta Math}. {\bf 88} (1952),  85--139.

\bibitem{CZ} A.P. Calder\'{o}n and A. Zygmund, On singular integrals,
    {\it Amer. J. Math}. {\bf 78} (1956),  289--309.

\bibitem{CLRT} J. Canto, K. Li, L. Roncal and O. Tapiola,
$C^p$ estimates for rough homogeneous singular integrals and sparse forms,  Available
at arXiv: 1909.08344 (2019).


\bibitem{CD} Y. Chen and Y. Ding, $L^{p}$ bounds for the commutators of singular integrals and maximal singular integrals with rough kernels,  {\it Trans. Amer. Math. Soc}. {\bf 367} (2015), 1585--1608.

\bibitem{Christ1} M. Christ, Weak type (1,1) bounds for rough operators,
    {\it Ann. of Math}. {\bf 128} (1988),  19--42.

\bibitem{Chr} M. Christ, A $T(b)$ theorem with
    remarks on analytic capacity and the Cauchy integral,
    {\it Colloq. Math}. {\bf 60/61} (1990),  601--628.


\bibitem{Christ2} M. Christ and J.L. Rubio de Francia, Weak type (1,1) bounds for rough operators. II,
    {\it Invent. Math}. {\bf 93} (1988),  225--237.


\bibitem{CW1} R. R. Coifman and G. Weiss, \emph{Analyse
    harmonique non-commutative sur certains espaces
    homog\`enes. \'Etude de certaines int\'egrales
    singuli\`eres}, Lecture Notes in Math. {\bf 242},
    Springer-Verlag, Berlin, (1971).


\bibitem{CCDO} J.M. Conde-Alonso, A. Culiuc, F. Di Plinio and Y. Ou, A sparse domination principle for rough singular integrals, {\it Anal. PDE}. {\bf 10} (2017),  1255--1284.

\bibitem{DL} Y. Ding and X.D. Lai, Weak type (1,1) bound criterion for singular integrals with rough kernel and its applications,  {\it Trans. Amer. Math. Soc}. {\bf 371} (2019),  1649--1675.


\bibitem{DLu} Y. Ding and S. Lu, Weighted norm inequalities for fractional integral operators with rough kernel. {\it Canad. J. Math}. {\bf 50} (1998), 29--39.

\bibitem{sato2} Y. Ding and S. Sato, Singular integrals on product homogeneous groups, {\it Integr. Equ. Oper. Theory}. {\bf 76} (2013), 55-79.


\bibitem{Duo1986} J. Duoandikoetxea, Multiple singular integrals and maximal functions along hypersurfaces,
{\it Ann. Inst. Fourier (Grenoble)}. {\bf 36} (1986), 185--206.

\bibitem{D} J. Duoandikoetxea, weighted norm inequalities for homogeneous singular integrals,  {\it Trans. Amer. Math. Soc}. {\bf 336} (1993), 869--880.

\bibitem{DR} J. Duoandikoetxea and J.L. Rubio de Francia, Maximal and singular integral operators via Fourier transform estimates,
    {\it Invent. Math}. {\bf 84} (1986),  541--561.


\bibitem{DGKLWY} X.T. Duong, R. Gong, M.S. Kuffner, J. Li, B.D. Wick and D. Yang, Two weight commutators on spaces of homogeneous type and applications. Available at arXiv: 1809.07942 (2018).


\bibitem{FP} D. Fan and Y. Pan, Singular inegrals operators with rough kernels supported by subvarieties,  {\it Amer. J. Math}. {\bf 119} (1997),  799--839.

\bibitem{FPY} D. Fan, Y. Pan and D. Yang, A weighted norm inequality for rough singular integrals,  {\it Tohoku Math. J}. {\bf 51} (1999),  141--161.

\bibitem{RF1981} R. Fefferman, Singular integral on product domains, {\it Bull. Amer. Math. Soc}. {\bf4} (1981), 195--201.

\bibitem {FoSt} G.B. Folland and E.M. Stein, Hardy Spaces on Homogeneous Groups, {\it Princetion University Press, Princeton, N. J}. (1982).

\bibitem{GS} L. Grafakos and A. Stefanov, $L^{p}$ bounds for singular integrals and maximal singular integrals with rough kernels, {\it Indiana Univ. Math. J}. {\bf 47} (1998),  455--469.

\bibitem{Hof} S. Hofmann, Weak (1,1) boundedness of singular integrals with nonsmooth kernel, {\it Proc. Amer. Math. Soc}. {\bf 103} (1988),  260--264.

\bibitem{H} T. Hyt\"onen, The sharp weighted bound for general Calder\'{o}n-Zygmund operators, {\it Ann. of Math}. {\bf 175} (2012),  1473--1506.

\bibitem{HK} T. Hyt\"onen and A. Kairema, Systems of
    dyadic cubes in a doubling metric space, {\it Colloq.
    Math}. {\bf 126} (2012),  1--33.

\bibitem{HPR} T. Hyt\"onen, C. P\'{e}rez and E. Rela, Sharp reverse H\"{o}lder property for $A_{\infty}$ weights on spaces of homogeneous type, {\it J. Funct. Anal}. {\bf 263} (2012), 3883-3899.

\bibitem{HRT} T. Hyt\"onen, L. Roncal and O. Tapiola, Quantitative weighted estimates for rough homogeneous singular integrals, {\it Israel J. Math}. {\bf 218} (2017),  133--164.

\bibitem{KLPW}
     A. Kairema, J. Li, C. Pereyra and L. A. Ward,  Haar bases on quasi-metric measure spaces, and dyadic structure theorems for function spaces on product spaces of homogeneous type, {\it J. Funct. Anal.} {\bf 271} (2016), 1793--1843.

\bibitem{La01}
   G. Karagulyan and  M.T. Lacey, On logarithmic bounds of maximal sparse operators,
{\it Math. Z.} {\bf294} (2020), no. 3-4, 1271--1281.

\bibitem{La02}
     R. Kesler, M. Lacey and D. Mena, Sparse bounds for the discrete spherical maximal functions, {\it Pure Appl. Anal.} {\bf2} (2020), no. 1, 75--92.


\bibitem{La}
     M.T. Lacey, An elementary proof of the $A_{2}$ bound, {\it Israel J. Math}. {\bf 217} (2017), 181--195.



\bibitem{Lerner1}
A.K. Lerner, On an estimate of Calder\'{o}n-Zygmund operators by dyadic positive operators, {\it J. Anal. Math}. {\bf 121} (2013), 141--161.



\bibitem{Lerner2}
A.K. Lerner, A simple proof of the $A_{2}$ conjecture, {\it Int. Math. Res. Not. IMRN}. (2013), 3159--3170.


\bibitem{Ler1}
A.K. Lerner, On pointwise estimates involving sparse operators, {\it New York J. Math}. {\bf 22} (2016), 341--349.


\bibitem{Lerner4}
A.K. Lerner, A note on weighted bounds for rough singular integrals, {\it C. R. Acad. Sci. Paris, Ser. I}. {\bf 356} (2018), 77--80.

\bibitem{Lerner3}
A.K. Lerner,  A weak type estimate for rough singular integrals, {\it Rev. Mat. Iberoam}. {\bf 35} (2019),  1583--1602.

\bibitem{LN}
A.K. Lerner and F. Nazarov, Intuitive dyadic calculus:  the basics, {\it Expo. Math}. {\bf 37} (2019), 225-265.

\bibitem{LPRR}
K. Li, C. P\'{e}rez, I.P. Rivera-R\'{i}os and L. Roncal, Weighted Norm Inequalities for Rough Singular Integral
Operators, {\it J. Geom. Anal}. {\bf47} (2019),  2526--2564.


\bibitem{LMW}
F. Liu, S. Mao and H. Wu, On rough singular integrals related to homogeneous mappings, {\it Collect. Math}. {\bf 67} (2016), 113--132.


\bibitem{Lorist}
E. Lorist, On pointwise $\ell^{r}$-sparse domination in a space of homogeneous type, {\it Journal of Geometry Analysis} (2020).


\bibitem{Moen} %
    K. Moen, %
    Sharp weighted bounds without testing or extrapolation, %
    {\it Archiv der Mathematik}.  {\bf 99} (2012), 457--466.

\bibitem{NRV}
F. Nazarov, A. Reznikov, and A. Volberg, The proof of $A_2$ conjecture in a geometrically doubling metric space, {\it Indiana Univ. Math. J}. {\bf62} (2013), 1503--1533.

\bibitem{PRR}  C. P\'erez, I.P. Rivera-R\'ios, L. Roncal, $A_1$ theory of weights for rough homogeneous singular integrals and commutators, {\it Ann. Sc. Norm. Super. Pisa Cl. Sci}. {\bf 19} (2019), 169--190.

\bibitem{sato} S. Sato, Estimates for singular integrals on homogeneous groups, {\it J. Math. Anal. Appl}. {\bf 400} (2013), 311--330.

\bibitem{SawW} E. Sawyer and R. L. Wheeden, Weighted
    inequalities for fractional integrals on Euclidean and
    homogeneous spaces, {\it Amer. J. Math}. {\bf 114} (1992),
    813--874.

  \bibitem{Se} A. Seeger, Singular integral operators with rough convolution kernels, {\it J. Amer. Math. Soc}. {\bf 9} (1996),
    95--105.

\bibitem{SW} E.M. Stein and G. Weiss, Interpolation of operators with change of measures, {\it Trans. Amer. Math. Soc}. {\bf 87} (1958),
    159--172.

\bibitem{s93} E. M. Stein, Harmonic Analysis:
Real-variable Methods, Orthogonality, and Oscillatory Integrals,
{\it Princeton, NJ: Princeton University Press}, (1993).

\bibitem{Tao} T. Tao, The Weak-type (1,1) of $L{\rm log}L$ Homogeneous Convolution Operator, {\it Indiana Univ. Math. J}. {\bf 48} (1999), 1547-1584.


\end{thebibliography}
\end{document}